\documentclass[a4paper,11pt,reqno]{article}
\usepackage[export]{adjustbox}
\usepackage{silence}
\usepackage{fullpage}
\usepackage[normalem]{ulem}
\usepackage[table]{xcolor} 
\usepackage{multirow}
\usepackage{mathtools}
\usepackage{amsmath}
\usepackage{adjustbox}
\usepackage{amsfonts}
\usepackage{amssymb}
\usepackage{multirow}
\usepackage{array}
\usepackage{makeidx}
\usepackage{xspace}
\usepackage{esvect}
\usepackage[most]{tcolorbox}
\usepackage{cite}
\usepackage{xspace} 
\usepackage{float}
\usepackage{tikz}
\usepackage{placeins}
\usepackage{babel}
\usepackage[caption=false]{subfig}
\usepackage{graphicx}
\usepackage{colortbl} 
\definecolor{lightgray}{rgb}{0.9,0.9,0.9} %
\usepackage[left=1.3cm,right=1.3cm,top=1.5cm,bottom=1.5cm]{geometry}
\usepackage{amsthm,amssymb,mathrsfs, pstricks,booktabs}
\allowdisplaybreaks
\newtheorem{thm}{Theorem}[section]

\newtheorem{lem}{Lemma}[section]

\usepackage{latexsym}
\usepackage{graphicx}
\usepackage{epstopdf}
\usepackage{play}
\usepackage{epsfig}
\usepackage[nottoc]{tocbibind}

\usepackage[utf8]{inputenc}
\usepackage[version=4]{mhchem}
\usepackage{stmaryrd}
\usepackage{bbold}

\usepackage{sectsty}
\sectionfont{\centering}
\makeatletter
\renewenvironment{proof}[1][\proofname]{%
  \par\pushQED{\qed}\normalfont%
  \topsep6\p@\@plus6\p@\relax
\trivlist\item[\hskip\labelsep\bfseries#1\@addpunct{.}]%
  \ignorespaces
}{%
  \popQED\endtrivlist\@endpefalse
}
\makeatother
\usepackage{blindtext}

\usepackage{hyperref}
\hypersetup{
    linkcolor=black,
    filecolor=black,      
    urlcolor=black,
    pdfpagemode=FullScreen,}
\begin{document}
\begin{center}
    \textbf{\large INVARIANCE ANALYSIS, SYMMETRY REDUCTION AND CONSERVATION LAWS FOR A BIOLOGICAL POPULATION MODEL IN POROUS MEDIA}
\end{center}
\begin{center}
\author[{\small \textbf{Urvashi Joshi}\textsuperscript{1}, \textbf{Aniruddha Kumar Sharma}\textsuperscript{2}, \textbf{Rajan Arora}\textsuperscript{3}\\ \textsuperscript{1,2,3} \text{Department of Applied Mathematics and Scientific Computing,}\\ \text{Indian Institute of Technology, Roorkee, India}}\\
\end{center}
\begin{center}
\textsuperscript{1}urvashi\textunderscore j@amsc.iitr.ac.in,
\textsuperscript{2}aniruddha\textunderscore s@as.iitr.ac.in, \\\textsuperscript{3}rajan.arora@as.iitr.ac.in.\textbf{(Corresponding Author)}
\end{center}
\section*{Abstract}
\quad \quad This research paper talks about using complex mathematical tools to study and figure out the behavior of biological populations in porous media. Porous media offer a unique environment where various factors, including fluid flow and nutrient diffusion, significantly influence population dynamics. The theory of Lie symmetries is used to find inherent symmetries in the governing equation of the population model, helping to find conservation laws and invariant solutions. The derivation and analysis of the optimal system provide insights into the most influential parameters affecting population growth and distribution. Furthermore, the study explores the construction of invariant solutions, which aid in characterizing long-term population behavior. The article concludes with the non-linear self-adjointness property and conservation laws for the model.

\textbf{Keywords:}\hspace{0.1cm}Biological population model, Porous media, Lie symmetric analysis, Symmetric group, Optimal subalgebra, Parameter analysis, Self-adjointness, Conservation laws.
\section{Introduction}
\quad \quad In many scientific and technical domains, including biology, fluid dynamics, heat transfer, structural mechanics, electromagnetics, image processing, and optimization, partial differential equations (PDEs), are essential tools for modeling and resolving a broad variety of real-world issues. Mathematical models in biology are utilized to depict and examine several biological phenomena, including population dynamics, disease transmission, and cellular function. These models are commonly constructed using ordinary differential equations (ODEs) and partial differential equations (PDEs), which are employed to depict the rates of change of biological variables over time and space.

Lie symmetry methods (\cite{bluman1989potential,ibragimov1995crc,olver1993applications,ovsiannikov2014group}) have proven effective in addressing a diverse array of issues in physics, engineering, and biology, such as the study of population dynamics in different contexts (\cite{maurya2024symmetry,simon2023optimal,bira2016application,bira2018exact,sharma2023invariance}). The existence of symmetries in the fundamental equation is the foundation of these methods, which can be utilized to simplify the problem and provide exact solutions. The model takes into account the impact of regional heterogeneity, diffusion, and nonlinear interactions among microbes, which are typical characteristics of population dynamics in the real world. The Lie symmetry approach is employed to ascertain the symmetries of the model and simplify the problem into ordinary differential equations (ODEs) that may be solved by analytical means. The resultant solutions offer vital insights into the population's behavior.

Finding all group-invariant solutions to a PDE using Lie algebra is a challenging task due to the number of subalgebras and symmetry reductions. Ovsiannikov \cite{ovsiannikov2014group} and Olver (\cite{olver1993applications,olver1987group}) introduced an ideal system for the Lie subalgebra, while Petera et al. \cite{patera1975continuous} advanced the method to analyze various examples of an optimal system for the Lie group in mathematical physics. Recently, researchers have applied this refined and expanded approach to develop optimal systems essential to finding invariant solutions to given PDE systems (\cite{devi2021optimal,singh20222+,chauhan2020lie,bobylev2020group}).

Non-linear self-adjointness property is crucial in deriving conservation laws for differential equations, particularly in non-linear systems (\cite{gandarias2014nonlinear,ibragimov2013nonlinear,zhang2017conservation}). A self-adjoint equation can be transformed into its adjoint without changing form, aiding in the identification of conserved quantities like energy or momentum. This concept extends classical approaches, allowing for the construction of conservation laws in cases where traditional symmetries are absent.
\subsection{Governing Model}
\quad \quad Mathematical modeling is crucial for regulating migration and understanding population density. The coordinate $\Vec{\text{x}}=(x,y)$ in the spatial domain $\mathbb{C}$ and the temporal domain $t$ play three essential roles in describing the spread of a biological species in the region. The functions in the model are denoted as population density $\phi(\Vec{\text{x}},t)$, diffusion velocity $v(\Vec{\text{x}},t)$, and population supply $f(\Vec{\text{x}},t)$. The entities $\phi$, $v$, and $f$ are required to comply with the population balance rule in the following manner for normal sub-region $\mathbb{U}$ of $\mathbb{C}$ and at time $t$
\begin{equation}\label{p}
\frac{d}{dt} \int_{\mathbb{U}} \phi dV + \int_{\partial \mathbb{U}} \phi \Vec{v}\cdot \hat{n} dA=\int_{\mathbb{U}} f dV \text{,}
\end{equation}
where $\hat{n}$ denotes the outward unit normal to the boundary $\partial \mathbb{U}$ of $\mathbb{U}$.
From Eq. \eqref{p}, the supply rate of individuals to $\mathbb{U}$ must satisfy the change in the population of $\mathbb{U}$ and
the departure rate of individuals from $\mathbb{U}$ over its boundary.

Now, we assume $f=f(\phi)$ and $\Vec{v}=-\lambda(\phi) \Delta \phi$, where $\Delta$ denotes the Laplace operator, with the property
that $\lambda^{\prime}(\phi)=0$ for $\phi=0$ and $\lambda^{\prime}(\phi)>0$ for $\phi>0$. Then, the population density $\phi$ involves a 2D non-linear degenerate parabolic PDE as follows:
\begin{equation}
\frac{\partial \phi}{\partial t}=\frac{\partial^{2} \lambda(\phi)}{\partial x^{2}}+\frac{\partial^{2} \lambda(\phi)}{\partial y^{2}}+f(\phi)\text{,}\hspace{0.2cm} t \geq 0 \hspace{0.2cm} \text{and} \hspace{0.2cm} x\text{,}\hspace{0.1cm}y \in R \text{.}
\end{equation}

Gurney and Nisbet \cite{gurney1975regulation} used $\lambda(\phi)$ as a particular case to show the animal population. Most migrations occur when young animals seek breeding sites or mature animals invade new territories, often moving toward areas with lower population density. Movement is typically faster in regions with higher population density. With $\lambda(\phi)=\phi^{2}$, the provided model generates Eq. \eqref{y}, a two-dimensional non-linear degenerate parabolic partial differential equation:
\begin{equation}\label{y}
\frac{\partial \phi}{\partial t}=\frac{\partial^{2} \phi^{2}}{\partial x^{2}}+\frac{\partial^{2} \phi^{2}}{\partial y^{2}}+f(\phi)\text{,}
\end{equation}
\begin{equation}
\phi_{t}-2 \phi_{x}^2-2 \phi  \phi_{x x}-2 \phi_{y}^2-2 \phi  \phi_{y y}-f(\phi)=0\text{.}
\end{equation}
\renewcommand{\labelenumi}{\alph{enumi})}
\begin{enumerate}
\item $f(\phi)=h \phi \quad \text{;} \quad h=$ constant (Malthusian law)\cite{gurtin1977diffusion}.
\item $f(\phi)=h_{1} \phi-h_{2} \phi^{2} \quad \text{;}\quad h_{1}\text{,} h_{2}=$ positive constant (Verhulst law)(\cite{gurtin1977diffusion,sharma2024study}).
\item  $f(\phi)=-h \phi^{\theta}\quad \text{;} \quad (h \geq 0 \text{,}\hspace{0.1cm}0<\theta<1)$ Porous media (\cite{bear1972flow,okubo1980diffusion}).
\end{enumerate}

The Malthusian growth model, introduced by Thomas Malthus in 1798, suggests that the population grows exponentially while the food supply increases arithmetically, leading to a potential mismatch that can cause famine, poverty, and conflict. In contrast, the Verhulst Law, or logistic equation, developed by Pierre-François Verhulst in the 1830s, models population growth as limited by available resources. Unlike the Malthusian model, which assumes unlimited growth, the logistic equation accounts for resource constraints, making it widely applicable in ecology and population dynamics
 (\cite{bacaer2011verhulst,verhulst1845recherches}).

The equation was subsequently found and widely disseminated by Raymond Pearl and Lowell Reed during the early 20th century. This paper will examine the model in porous 
medium, which assumes $f(\phi)=-h 
 \phi^{\theta}$. Here, $h$ is the population growth rate and $\theta$ is a porous constant. Then we have the following equation:
\begin{equation}\label{u}
\Delta:\hspace{0.2cm} \phi_{t}-2 \phi_{x}^{2}-2 \phi  \phi_{x x}-2 \phi_{y}^{2}-2 \phi  \phi_{y y}+h \phi^{\theta}=0 \text{,}
\end{equation}
where $\hspace{0.2cm} 0 < \theta < 1, $ $h\geq 0$ and $t>0$, $x,y\in R $.
\subsection{Outline}
\quad \quad The paper aims to achieve three main objectives, firstly, to determine an optimal system, secondly, to search for multiple new exact solutions, and thirdly, to derive conservation laws. Section 2 will involve identifying the symmetries of Eq. \eqref{u} using the Lie group technique. In section 3, we identify the symmetry group of the model. In section 4, we establish the optimal system of subalgebra. In section 5, we find group-invariant solutions using optimal system and parameter analysis. In section 6, we demonstrate the non-linear self-adjointness property for Eq. \eqref{u}. In section 7, we establish the conservation laws using the non-linear self-adjointness property for Eq. \eqref{u}. The conclusions will provide insight into the physical meaning of the results.
\section{Lie Point Symmetries}
\quad \quad The Lie group of transformation with $\varepsilon$ as a small parameter, acting on the dependent variable $\phi$, independent variables $x$, $y$, and $t$ for Eq. \eqref{u}, is denoted as  
\begin{equation}
\begin{split}
&\hat{x}=x+\varepsilon  \xi_{1}(x, y, t, \phi)+O\left(\varepsilon^{2}\right)\text{,}\\
&\hat{y}=y+\varepsilon \xi_{2}(x, y, t, \phi)+O\left(\varepsilon^{2}\right)\text{,}\\
&\hat{t}=t+\varepsilon  \xi_{3}(x, y, t, \phi)+O\left(\varepsilon^{2}\right)\text{,}  \\
&\hat{\phi}=\phi+\varepsilon  \xi_{4}(x, y, t, \phi)+O\left(\varepsilon^{2}\right)\text{,} 
\end{split}
\end{equation}
where $\xi_{1}, \xi_{2}, \xi_{3}, \xi_{4}$ are the infinitesimals generator for Eq. \eqref{u} which are to be determined and ensure that the equation remains invariant.\\
Thus, the Lie algebra that is associated with this will be of the following form (\cite{olver1993applications,bluman1989potential,ibragimov1995crc,ovsiannikov2014group}):
\begin{equation}
\mathbb{Z}=\xi_{1}(x, y, t, \phi) \partial_{x}+\xi_{2}(x, y, t, \phi) \partial y+\xi_{3}(x, y, t, \phi) \partial_{t}+\xi_{4}(x, y, t, \phi) \partial_{\phi}\text{.}
\end{equation}\\
The invariance is ${Pr}^{(2)} \mathbb{Z}(\Delta)=0$, where $\Delta=0$ for Eq. \eqref{u} and ${Pr}^{(2)}$ is the second-order prolongation of $\mathbb{Z}$ for the test Eq. \eqref{u}.\\
We applied ${Pr}^{(2)} \mathbb{Z}$ to the model \eqref{u} in order to obtain an overdetermined system of the coupled PDEs.
\begin{equation}
{Pr}^{(2)}=  \mathbb{Z}+\xi_{4}^{t} \frac{\partial}{\partial \phi_{t}}+\xi_{4}^{x} \frac{\partial}{\partial \phi_{x}}+\xi_{4}^{x x} \frac{\partial}{\partial \phi_{x x}}+\xi_{4}^{y} \frac{\partial}{\partial \phi_{y}}+\xi_{4}^{y y} \frac{\partial}{\partial \phi_{y y}} +\xi_{4} \frac{\partial}{\partial \phi}\text{.}
\end{equation}
The invariant condition is obtained by applying the above prolongation to the model \eqref{u} as follows:
\begin{equation}
\xi_{4}^{t}-4 \phi_{x}  \xi_{4}^{x}-2 \phi  \xi_{4}^{x x}-2 \phi_{x x} \xi_{4}-4 \phi_{y} \xi_{4}^{y}-2 \phi  \xi_{4}^{y y}-2 \phi_{y y} \xi_{4}=0\text{.}
\end{equation}
The following system of PDEs has been generated using Maple, a computer algebra software:
\begin{equation}\label{k}
\begin{split}
&\left(\xi_{1}\right)_{t}=0=\left(\xi_{1}\right)_{\phi}\text{,} \quad \left(\xi_{1}\right)_{x}=\frac{\left(\xi_{3}\right)_{t}(\theta-2)}{2(\theta-1)}\text{,} \quad \left(\xi_{1}\right)_{y}=-\left(\xi_{2}\right)_{x}\text{,} \left(\xi_{2}\right)_{t}=0=\left(\xi_{2}\right)_{\phi}\text{,} \quad \\ &\left(\xi_{2}\right)_{y}=\frac{(\theta-2)\left(\xi_{3}\right)_{t}}{2(\theta-1)}\text{,}  \quad \left(\xi_{2}\right)_{x x}=0\text{,} \left(\xi_{3}\right)_{tt}=0=\left(\xi_{3}\right)_{u}=\left(\xi_{3}\right)_{x}=\left(\xi_{3}\right)_{y}\text{,} \quad \xi_{4}=-\frac{u \left(\xi_{3}\right)_{t}}{(\theta-1)}\text{.}
\end{split}
\end{equation}
After solving the system of PDEs in Eq. \eqref{k}, we obtain the required infinitesimal generators as follows:
\begin{equation}
\begin{split}
&\xi_{1}=\frac{c_{1}(\theta-2) x}{2(\theta-1)}+c_{3} y+c_{4}\text{,} \\ 
&\xi_{2}=-c_{3} x+\frac{c_{1}(\theta-2) y}{2(\theta-1)}+c_{5}\text{,} \\ 
&\xi_{3}=c_{1} t+c_{2}\text{,}\quad \xi_{4}=-\frac{c_{1} \phi}{(\theta-1)}\text{,}
\end{split}
\end{equation}
where $c_{i}$'s ; $i=1,2,3,4,5$ are the arbitrary constants.\\
Following the Lie symmetry method (\cite{olver1993applications,bluman1989potential,ibragimov1995crc,ovsiannikov2014group}), we get the Lie algebra of symmetries for Eq. \eqref{u} as follows:
\begin{equation}\label{n}
\begin{split}
& X_{1}= \gamma \partial_{x}+ \gamma \partial_{y}+t  \partial_{t}-\frac{\phi}{(\theta-1)} \partial \phi\text{,} \\
& X_{2}=\partial_{t},\quad X_{3}=y \partial x-x  \partial y\text{,}\quad X_{4}=\partial_{x}\text{,}\quad X_{5}=\partial y \text{,}
\end{split}
\end{equation}
where $\gamma=\frac{(\theta-2)}{2(\theta-1)} \hspace{0.2cm}\text{;}  \hspace{0.2cm} \theta \in(0,1)$.\\
Now, we obtain Table \ref{Table 1} of commutation, which has entries like $[X_i,X_j]=X_{i}X_{j}-X_{j}X_{i}$.
\renewcommand{\arraystretch}{1.5}
\begin{table}[H]
    \centering
    \begin{tabular}{|l|l|l|l|l|l|}
    \hline
    \rowcolor{gray!30} 
        $\star$ & $X_1$ & $X_2$ & $X_3$ & $X_4$ & $X_5$ \\ \hline 
        $X_1$ & 0 & $-X_2$ & 0 & $-\gamma X_4$ & $-\gamma X_5$ \\ \hline
 \rowcolor{gray!30} 
       $X_2$ & $X_2$ & 0 & 0 & 0 & 0 \\ \hline
        $X_3$ & 0 & 0 & 0 & $X_5$ & $-X_4$ \\ \hline
         \rowcolor{gray!30} 
        $X_4$ & $\gamma X_4$ & 0 &$-X_5$ & 0 & 0 \\ \hline
        $X_5$ & $\gamma X_5$ & 0 & $X_4$ & 0 & 0 \\ \hline
    \end{tabular}
    \caption{Commutation Table}
    \label{Table 1}
\end{table}
An infinite group of transformations of Eq. \eqref{u} is generated by the infinite-dimensional Lie algebra spanned by Eq. \eqref{n}. It is necessary to combine all equivalent subalgebras into a single category and select an agent from each category. The combination of agents is referred to as an optimal system.\\
Hence, each infinitesimal of Eq. \eqref{u} is appropriately represented as a linear combination of $X_{i}$ and $\mathcal{X}$ is expressed as:
$$
\mathcal{X}=\alpha_{1} X_{1}+\alpha_{2} X_{2}+\alpha_{3} X_{3}+\alpha_{4} X_{4}+\alpha_{5} X_{5}\text{.}
$$
\section{Symmetry Group of Model}
\quad \quad Let the one-parameter group $\tilde{G}_{i}$ generated by $X_{i}\hspace{0.2cm} \text{;} \hspace{0.2cm} i=1,2, \ldots, 5$\text{,}  be
\begin{equation}
\tilde{G}_{i}:(x, y, t, \phi) \rightarrow(\tilde{x}, \tilde{y}, \tilde{t}, \tilde{\phi}) \text {, }
\end{equation}
that produces some invariant solutions from the known ones. We need to solve the following system of ODEs:
\begin{equation}
\begin{aligned}
& \frac{d}{d \varepsilon}(\tilde{x}, \tilde{y}, \tilde{t}, \tilde{\phi})=\Gamma(\tilde{x}, \tilde{y}, \tilde{t}, \tilde{\phi})\hspace{0.1cm}\text { with } \hspace{0.1cm}\left.(\tilde{x}, \tilde{y}, \tilde{t}, \tilde{\phi})\right|_{\varepsilon=0}=(x, y, t, \phi) \text{,}
\end{aligned}
\end{equation}
where $\varepsilon$ is an arbitrary real parameter and \begin{equation}
\Omega=\xi_{1} \phi_{x}+\xi_{2} \phi_{y}+\xi_{3} \phi_{t}+\xi_{4} \phi \text {. }
\end{equation}
According to different $\xi_{1}, \xi_{2}, \xi_{3} \hspace{0.1cm} \text{and} \hspace{0.1cm} \xi_{4}$, we have following groups:
\begin{equation}
\begin{aligned}
& \tilde{G}_{1}:\left(x e^{\gamma \varepsilon}, y e^{\gamma \varepsilon}, t e^{\varepsilon}, \phi e^{-\frac{\varepsilon}{\theta-1}}\right) \rightarrow \text{Dilation in}\hspace{0.1cm} x,y,t \hspace{0.1cm} \text{and} \hspace{0.1cm} \phi \text{,}\\
&\tilde{G}_{2}:(x, y, t+\varepsilon, \phi)\rightarrow \text{Translation in time}\hspace{0.1cm} t \text{,}\\
& \tilde{G}_{3}:(x+\varepsilon y, y-\varepsilon x, t, \phi) \rightarrow \text{Rotation}\text{,}\\
&\tilde{G}_{4}:(x+\varepsilon, y, t, \phi) \rightarrow \text{Translation in} \hspace{0.1cm} x \text{,}\\
& \tilde{G}_{5}:(x, y+\varepsilon, t, \phi) \rightarrow \text{Translation in} \hspace{0.1cm} y \text{.}
\end{aligned}
\end{equation}
The right-hand side gives the transformed point $\exp \left(\varepsilon X_{i}\right)(x, y, t, \phi)$ $=(\tilde{x}, \tilde{y}, \tilde{t}, \tilde{\phi})$. Since $\tilde{G}_{i}$ is a symmetry group and if $\phi=f(x, y, t)$ is a solution of the model \eqref{u}, the corresponding new solutions can be given as:
$$
\begin{array}{ll}
\tilde{G}_{1}: \phi=f\left(x e^{-\gamma \varepsilon}, y e^{-\gamma \varepsilon}, t e^{-\varepsilon}\right) \cdot e^{(\varepsilon /(\theta-1))} \text{,}\\
\tilde{G}_{2}: \phi=f(x, y, t-\varepsilon) \text{,} \\
 \tilde{G}_{3}: \phi=f(x-\varepsilon y, y+\varepsilon x, t) \text{,}\\
\tilde{G}_{4}: \phi=f(x-\varepsilon, y, t) \text{,} \\
 \tilde{G}_{5}: \phi=f(x, y-\varepsilon, t) \text{.}
\end{array}
$$
\section{Classification of Optimal Subalgebras}
\quad \quad We will develop a one-dimensional optimal system for Eq. \eqref{u} using the commutator and adjoint tables. An optimal system consists of subalgebras where each subalgebra is equivalent to a unique member under the adjoint representation. Ovsiannikov \cite{ovsiannikov2014group} and Olver (\cite{olver1993applications},\cite{olver1987group}) introduced this approach, and Petera et al. \cite{patera1975continuous} further advanced it by applying it to various examples in mathematical physics. The systematic algorithm will guide the construction of the optimal system.\\
The set of all infinitesimal symmetries of the partial differential equations forms a Lie algebra satisfying skew-symmetry, with each diagonal entry zero under the Lie bracket.
$$
\left[X_{i}, X_{j}\right]=X_{i}  X_{j}-X_{j} X_{i} \text{.}
$$
By using a commutator table, we obtain the adjoint representation table. In the adjoint table, the entry at $(i, j)$ position is
$$
\operatorname{Ad}\left(\exp \left(\varepsilon X_{i}\right) X_{j}\right)=X_{j}-\varepsilon\left[X_{i}, X_{j}\right]+\frac{1}{2 !} \varepsilon^{2}\left[X_{i},\left[X_{i}, X_{j}\right]\right]-\dots \hspace{0.2cm} \text{.}
$$
An invariant $\pi$ is a real function on the Lie algebra $L^{5}$, i.e., $\pi(\operatorname{Ad_{g}}(x))=\pi(x)$ for all $x \in L^{5}$, generated by $X_{i}$.
\subsection{Calculation of Invariants}
\quad \quad To get the one-dimensional optimal system of $L^{5}$  Lie algebra, we have to consider the invariant function for the selection of representative elements as follows:
$$
S=\sum_{i=1}^{5} \alpha_{i} X_{i} \text { and } M=\sum_{j=1}^{5} \beta_{j} X_{j} \text {. }
$$
The adjoint action $\operatorname{Ad}(\exp (\varepsilon M))$ on $S$ is given by
$$
\begin{aligned}
\operatorname{Ad}(\exp (\varepsilon M) S) & =S-\varepsilon[M, S]+\frac{1}{2 !} \varepsilon^{2}[M,[M, S]] \ldots=e^{-\varepsilon M} S e^{\varepsilon M} \text{.}\\
&= \left(\alpha_{1} X_{1}+\alpha_{2} X_{2}+\ldots+\alpha_{5} X_{5}\right)-\varepsilon\left[\beta_{1} X_{1}+\ldots+\beta_{5} X_{5}, \alpha_{1} X_{1}+\ldots+\alpha_{5} X_{5}\right]+O\left(\varepsilon^{2}\right) \text{.} \\
&=  \left(\alpha_{1} X_{1}+\alpha_{2} X_{2}+\ldots+\alpha_{5} X_{5}\right)-\varepsilon\left[\lambda_{1} X_{1}+\lambda_{2} X_{2}+\ldots+\lambda_{5} X_{5}\right]+O\left(\varepsilon^{2}\right) \text{,}
\end{aligned}
$$
where $\lambda_{i}= \lambda_{i}\left(\alpha_{1}, \alpha_{2}, \ldots, \alpha_{5}, \beta_{1}, \beta_{2}, \ldots, \beta_{5}\right), \hspace{0.1cm} i=1,2,3,4,5$ which can be calculated by commutator table and for invariance,
\begin{equation}\label{q}
\pi\left(\alpha_{1}, \alpha_{2}, \ldots, \alpha_{5}\right)=\pi\left(\alpha_{1}-\varepsilon \lambda_{1}, \alpha_{2}-\varepsilon \lambda_{2}, \ldots, \alpha_{5}-\varepsilon \lambda_{5}\right)\text{.} 
\end{equation}
The following is obtained by applying Taylor's expansion to the right-hand side of Eq. \eqref{q}:
\begin{equation}\label{l}
\lambda_{1}  \frac{\partial \pi}{\partial \alpha_{1}}+\lambda_{2}  \frac{\partial \pi}{\partial \alpha_{2}}+\ldots+\lambda_{5}  \frac{\partial \pi}{\partial \alpha_{5}}=0\text{,} 
\end{equation}
$\text{where}\hspace{0.2cm}\lambda_{1}=0\text{,}\quad \lambda_{2}=-\alpha_{2}\beta_{1}+\alpha_{1} \beta_{2}\text{,}\quad \lambda_{3}=0\text{,}\quad \lambda_{4}=\gamma \left(-\alpha_{4} \beta_{1} +\alpha_{1} \beta_{4} \right)-\alpha_{5}\beta_{3} +\alpha_{3} \beta_{5}\text{,}\quad \lambda_{5}=\gamma \left(-\alpha_{5}\beta_{1} +\alpha_{1} \beta_{5}\right)+\alpha_{4} \beta_{3}-\alpha_{3} \beta_{4}\text{.}$ Substituting these values in Eq. \eqref{l} and extracting the coefficients $\forall$ $\beta_{j}\hspace{0.1cm}\text{,}\hspace{0.1cm}  1 \leq j \leq 5$, the system of first-order PDEs can be obtained as follows:
\begin{equation}
\begin{split}
&\beta_{1}: -\alpha_{2}\frac{\partial \pi}{\partial \alpha_{2}}-\gamma \alpha_{4} \frac{\partial \pi}{\partial \alpha_{4}}-\gamma \alpha_{5} \frac{\partial \pi}{\partial \alpha_{5}}=0\text{,} \\
&\beta_{2} : \alpha_{1} \frac{\partial \pi}{\partial \alpha_{2}}=0\text{,}\\
&\beta_{3}:\alpha_{4}\frac{\partial \pi}{\partial \alpha_{5}}-\alpha_{5} \frac{\partial \pi}{\partial \alpha_{4}}=0 \text{,} \\
&\beta_{4} : \gamma  \alpha_{1} \frac{\partial \pi}{\partial \alpha_{4}}-\alpha_{3}  \frac{\partial \pi}{\partial \alpha_{5}}=0\text{,} \\
&\beta_{5} : \gamma  \alpha_{1}  \frac{\partial \pi}{\partial \alpha_{5}}+\alpha_{3} \frac{\partial \pi}{\partial \alpha_{4}}=0\text{.} 
\end{split}
\end{equation}
For the solution of the above system of PDEs, one can find that
\begin{equation}
\pi\left(\alpha_{1}, \alpha_{2}, \alpha_{3}, \alpha_{4}, \alpha_{5}\right)=\mathcal{M}\left(\alpha_{1}, \alpha_{3}\right)\text{,}
\end{equation}
where $\mathcal{M}$ can be chosen as an arbitrary function.
\subsection{Killing Form}
\begin{thm}
The Killing form associated with Lie algebra $L^{5}$ is 
\begin{equation}
    K(X, X)=\left(2 \gamma^{2}+1\right) \alpha_{1}^{2}-2 \alpha_{3}^{2}\text{.}
\end{equation}
\end{thm}
\begin{proof}
We assume the Lie algebra denoted as $\psi$ over the field $\mathbb{F}$. Utilizing the Lie bracket operation, every member $x$ of $\psi$ characterizes the adjoint endomorphism $\operatorname{Ad}(x)$ of $\psi$, where $\operatorname{Ad}(x)(y)=[x,y]$. By creating a symmetric bilinear form, the trace of the composition of two such endomorphisms yields a value in $\mathbb{F}$ called the Killing Form on the finite-dimensional Lie algebra $\psi$, $i.e$.,
$$K(x,y)=Tr(\operatorname{Ad}(x) \circ \operatorname{Ad}(y))\text{.} $$
The explicit expression of the Killing form connected with the Lie algebra $L^{5}$ is provided as follows:
$$K(X,X)=Trace(\operatorname{ad(X)} \circ \operatorname{ad(X)})\text{,} $$
where
$$\begin{aligned}
& \operatorname{Ad}(X)=\left[\begin{array}{ccccc}
0 & 0 & 0 & 0 & 0 \\
-\alpha_{2} & \alpha_{1} & 0 & 0 & 0 \\
0 & 0 & 0 & 0 & 0 \\
-\gamma \alpha_{4} & 0 & -\alpha_{5} & \gamma \alpha_{1} & \alpha_{3} \\
-\gamma \alpha_{5} & 0 & \alpha_{4} & -\alpha_{3} & \gamma \alpha_{1}
\end{array}\right]\text{,}
\end{aligned} $$ 
which on simplification, generates 
$K(X, X)=\left(2 \gamma^{2}+1\right) \alpha_{1}^{2}-2 \alpha_{3}^{2}$\text{.}
\end{proof}
\subsection{Construction of Adjoint Transformation Matrix}
\quad \quad It is now required to construct the general adjoint transformation matrix $\mathcal{A}$, which is composed of individual matrices of the adjoint actions $M_{1}, M_{2}, \ldots M_{5}$ with respect to $X_{1}, X_{2}, \ldots X_{5}$ to $\mathcal{A}$. Let $\varepsilon_{i}\text{,}  \hspace{0.1cm}i=1,2,3,4,5$ be real constants and $g=e^{\varepsilon_{i} X_{i}}$, then we get
\begin{alignat*}{2}
M_{1}= &
\begin{pmatrix}
1 & 0 & 0 & 0 & 0 \\
0 & e^{\varepsilon_{1}} & 0 & 0 & 0 \\
0 & 0 & 1 & 0 & 0 \\
0 & 0 & 0 & e^{\gamma \varepsilon_{1}} & 0 \\
0 & 0 & 0 & 0 & e^{\gamma \varepsilon_{1}}
\end{pmatrix}\text{,}  \qquad&
M_{2}= & 
\begin{pmatrix}
1 & -\varepsilon_{2} & 0 & 0 & 0 \\
0 & 1 & 0 & 0 & 0 \\
0 & 0 & 1 & 0 & 0 \\
0 & 0 & 0 & 1 & 0 \\
0 & 0 & 0 & 0 & 1
\end{pmatrix}\text{,} 
\end{alignat*}
\begin{alignat*}{3}
M_{3}= &
\begin{pmatrix}
1 & 0 & 0 & 0 & 0 \\
0 & 1 & 0 & 0 & 0 \\
0 & 0 & 1 & 0 & 0 \\
0 & 0 & 0 & \cos{\varepsilon_{3}} & -\sin{\varepsilon_{3}} \\
0 & 0 & 0 & \sin{\varepsilon_{3}} & \cos{\varepsilon_{3}}
\end{pmatrix}\text{,}  \qquad&
M_{4}= &
\begin{pmatrix}
1 & 0 & 0 & -\gamma \varepsilon_{4} & 0 \\
0 & 1 & 0 & 0 & 0 \\
0 & 0 & 1 & 0 & \varepsilon_{4} \\
0 & 0 & 0 & 1 & 0 \\
0 & 0 & 0 & 0 & 1
\end{pmatrix}\text{,}  \qquad&
     M_{5}= &
     \begin{pmatrix}
1 & 0 & 0 & 0 & -\gamma \varepsilon_{5} \\
0 & 1 & 0 & 0 & 0 \\
0 & 0 & 1 & -\varepsilon_{5} & 0 \\
0 & 0 & 0 & 1 & 0 \\
0 & 0 & 0 & 0 & 1
 \end{pmatrix}\text{.} 
\end{alignat*}
One can obtain the adjoint action of $X_{j}$ on $X_{i}$ from the adjoint representation table.
\renewcommand{\arraystretch}{1.5}
\begin{table}[H]
    \centering
    \begin{tabular}{|l|l|l|l|l|l|}
    \hline
     \rowcolor{gray!30} 
        $\operatorname{Ad}(\exp{\varepsilon(\star)}\star)$ & $X_1$ &$ X_2$ & $X_3$ & $X_4$ & $X_5$ \\ \hline
        $X_1$ & $X_1$ & $e^{\varepsilon}X_2$ & $X_3$ & $e^{\gamma \varepsilon} X_4$ & $e^{\gamma \varepsilon} X_5$ \\ \hline
         \rowcolor{gray!30} 
        $X_2$ & $X_1-\varepsilon X_2$ & $X_2$ & $X_3$ & $X_4$ & $X_5$ \\ \hline
        $X_3$ & $X_1$ & $X_2$ & $X_3$ & $X_4 \cos{\varepsilon}- X_5 \sin{\varepsilon}$ & $X_4 \sin{\varepsilon}+ X_5 \cos{\varepsilon}$ \\ \hline
         \rowcolor{gray!30} 
        $X_4$ & $X_1-\gamma \varepsilon X_4$ & $X_2$ & $X_3+\varepsilon X_5$ & $X_4$ & $X_5$ \\ \hline
        $X_5$ & $X_1-\gamma \varepsilon X_5$ & $X_2$ & $X_3-\varepsilon X_4$ & $X_4$ & $X_5$ \\ \hline
    \end{tabular}
    \caption{Adjoint Representation Table}
\end{table}
The system of algebraic equations is solved to derive an optimal Lie algebra system. The equivalent Lie subalgebras can be identified by applying adjoint action to the set of these Lie subalgebras. Let
$$
\mathcal{X}=\alpha_{1} X_{1}+\alpha_{2} X_{2}+\alpha_{3} X_{3}+\alpha_{4} X_{4}+\alpha_{5} X_{5} \text {,}
$$
where $\alpha_{1}\text{,}\hspace{0.1cm}  \alpha_{2}\text{,} \hspace{0.1cm} \alpha_{3}\text{,}\hspace{0.1cm}  \alpha_{4} \hspace{0.2cm} \text{and} \hspace{0.2cm} \alpha_{5}$ are the real constants.
Here, $\mathcal{X}$ can be written as a column vector with entries $\alpha_{1}, \alpha_{2}, \ldots, \alpha_{5}$. $\mathcal{A}\left(\varepsilon_{1}, \varepsilon_{2},\dots, \varepsilon_{5}\right)=M_{5} M_{4} M_{3} M_{2} M_{1}$, gives
$$
\mathcal{A}=\left[\begin{array}{ccccc}
1 & -\varepsilon_{2} e^{\varepsilon_{1}} & 0 & -\gamma \sigma_{1} e^{\gamma \varepsilon_{1}} & \gamma \sigma_{2} e^{\gamma \varepsilon_{1}} \\
0 & e^{\varepsilon_{1}} & 0 & 0 & 0 \\
0 & 0 & 1 & \sigma_{2} e^{\gamma \varepsilon_{1}} & \sigma_{1}e^{\gamma \varepsilon_{1}} \\
0 & 0 & 0 & e^{\gamma \varepsilon_{1}} \cos{\varepsilon_{3}} & -e^{\gamma \varepsilon_{1}} \sin{\varepsilon_{3}}  \\
0 & 0 & 0 & e^{\gamma \varepsilon_{1}} \sin{\varepsilon_{3}} & e^{\gamma \varepsilon_{1}} \cos{\varepsilon_{3}}
\end{array}\right]\text{,} 
$$
where $\sigma_{1}=\varepsilon_{4} \cos{\varepsilon_{3}}+\varepsilon_{5}\sin{\varepsilon_{3}} \quad \text{and} \quad \sigma_{2}=\varepsilon_{4} \sin{\varepsilon_{3}}-\varepsilon_{5}\cos{\varepsilon_{3}}$.
\renewcommand{\arraystretch}{1.5}
\begin{table}[H]
    \centering
    \begin{tabular}{|c|c|c|c|c|c|}\hline
    \multicolumn{6}{|c|}{Cofficients of $X_{i}$} \\\hline
     \rowcolor{gray!30} 
        $\operatorname{Ad}(\exp{(\varepsilon X_i)}X)$ & $X_1$ &$ X_2$ & $X_3$ & $X_4$ & $X_5$ \\ \hline
        $\operatorname{Ad}(\exp{(\varepsilon X_1)}X)$ & $\alpha_1$ & $e^{\varepsilon}\alpha_2$ & $\alpha_3$ & $e^{\gamma \varepsilon} \alpha_4$ & $e^{\gamma \varepsilon} \alpha_5$ \\ \hline
         \rowcolor{gray!30} 
        $\operatorname{Ad}(\exp{(\varepsilon X_2)}X)$ & $\alpha_1$ & $\alpha_2-\varepsilon \alpha_1$ & $\alpha_3$ & $\alpha_4$ & $\alpha_5$ \\ \hline
        $\operatorname{Ad}(\exp{(\varepsilon X_3)}X)$ & $\alpha_1$ & $\alpha_2$ & $\alpha_3$ & $\alpha_4 \cos{\varepsilon}+ \alpha_5 \sin{\varepsilon}$ & $\alpha_5 \cos{\varepsilon}-\alpha_4 \sin{\varepsilon}$ \\ \hline
         \rowcolor{gray!30} 
        $\operatorname{Ad}(\exp{(\varepsilon X_4)}X)$ & $\alpha_1$ & $\alpha_2$ & $\alpha_3$ & $\alpha_4- \gamma \varepsilon \alpha_1$ & $\alpha_5+\varepsilon \alpha_3$ \\ \hline
        $\operatorname{Ad}(\exp{(\varepsilon X_5)}X)$ & $\alpha_1$ & $\alpha_2$ & $\alpha_3$ & $\alpha_4-\varepsilon \alpha_3$ & $\alpha_5-\gamma \varepsilon \alpha_1$ \\ \hline
    \end{tabular}
    \caption{Table for Construction of Invariant Functions}
    \label{t2}
\end{table}
\begin{lem}
    $M=\alpha_1$ and $N=\alpha_3$ are invariants for $L^{5}$.
    \end{lem}
\begin{proof}
It is evident from the data in Table \ref{t2}.
     \end{proof}
\begin{lem}
Let
$$
P=\left\{\begin{array}{llll}
1 ; & \alpha_{1}^{2}+\alpha_{2}^{2}+\alpha_{3}^{2} \neq 0  \\
0 ; & \text { otherwise, }
\end{array}\right.
$$
be a real valued function of $L^5$. Then $P$ is invariant.
\end{lem}
\begin{proof}
    The coefficients of $X_{1}, X_{2}$ and $X_{3}$ remain unchanged under the action of $\operatorname{Ad}\left(\exp \left(\varepsilon_{i} X_{i}\right)\right)\hspace{0.1cm};\hspace{0.1cm} i=3,4,5$. Hence, it is enough to check the invariance of $P$ under the action of $\operatorname{Ad}\left(\exp \left(\varepsilon X_{1}\right)\right)$ and $\operatorname{Ad}\left(\exp \left(\varepsilon X_{2}\right)\right)$.\\
Let $\alpha_{i}^{*}\hspace{0.1cm} ;\hspace{0.1cm} i=1,2,3,4,5$ be the new transformed coefficients after the adjoint actions. With the action of $\operatorname{Ad}\left(\exp \left(\varepsilon X_{1}\right)\right), \operatorname{Ad}\left(\exp \left(\varepsilon X_{2}\right)\right)$ on $X$, we obtain
$$
\begin{aligned}
& \alpha_{1}^{*^{2}}+\alpha_{2}^{\star^{2}}+\alpha_{3}^{\star^{2}}=\alpha_{1}^{2}+\left(e^{\varepsilon} \alpha_{2}\right)^{2}+\alpha_{3}^{2}\hspace{0.1cm} \text{,} \\
& \alpha_{1}^{\star^{2}}+\alpha_{2}^{\star^{2}}+\alpha_{3}^{\star^{2}}=\alpha_{1}^{2}+\left(\alpha_{2}-\varepsilon \alpha_{1}\right)^{2}+\alpha_{3}^{2} \hspace{0.1cm}\text{.}
\end{aligned}
$$
It is clear that $\alpha_{1}^{*}=\alpha_{2}^{*}=\alpha_{3}^{*}=0$ iff $\alpha_{1}=\alpha_{2}=\alpha_{3}=0$.
\end{proof}
\begin{lem}
    $Q$ is invariant, where
$$
Q= \begin{cases}1 ; & \alpha_{1}^{2}+\alpha_{3}^{2}+\alpha_{4}^{2}+\alpha_{5}^{2} \neq 0 \\ 0, & \text { otherwise. }\end{cases}
$$
\end{lem}
\begin{proof}
    The proof is similar to the proof of above Lemma.
    \end{proof}
\begin{lem}\label{lem_1}
$\mathcal{X}=\sum_{i=1}^{5} \alpha_{i} X_{i}$ with $\alpha_{1}=\alpha_{3}=\alpha_{5}=0$. Suppose $\operatorname{Ad}\left(\exp \left(\varepsilon X_{i}\right) \mathcal{X}\right)=\sum_{j=1}^{5} \alpha_{j}^{*} X_{j}$ also satisfy $\alpha_{1}^{*}=\alpha_{3}^{*}=\alpha_{5}^{*}=0$ \quad;\quad for each $i=1,2, \ldots 5$. Then $R^{*}=R$, i.e., $R(\mathcal{X})=R(\operatorname{Ad_{g}}(\mathcal{X}))$, where
$$
R= \begin{cases}\operatorname{sgn}\left(\alpha_{4}\right) & ;\text { if } \alpha_{1}=\alpha_{3}=\alpha_{5}=0 \\ 0  & ;\text { otherwise. }\end{cases}
$$
\end{lem}
\begin{proof}
     The coefficient of $X_{4}$, i.e., $\alpha_{4}$ remains unchanged under the action of $\operatorname{Ad}\left(\exp \left(\varepsilon X_{i}\right)\right) , i=2$. Therefore, we investigate the invariance under the adjoint actions $\operatorname{Ad}\left(\exp \left(\varepsilon X_{j}\right)\right) , j=1,3,4,5$. For $j=3,4,5$, the conditional invariant $R=R^{\star}=0$. The $\operatorname{sgn}\left(\alpha_{4}\right)$ maps to $e^{\gamma \varepsilon} \alpha_{4}$ for $j=1$, which is positive, negative or zero, depending on the sign of $\alpha_{4}$.
     \end{proof}
\begin{lem}
$\mathcal{X}=\sum_{i=1}^{5} \alpha_{i} X_{i}$ with $\alpha_{1}=0$. Suppose $\operatorname{Ad}\left(\exp \left(\varepsilon X_{i}\right) \mathcal{X}\right)=\sum_{j=1}^{5} \alpha_{j}^{*} X_{j}$ also satisfies $\alpha_{1}^{*}$ for each $i=1,2,3,4,5$. Then $S^{*}=S$, i.e., $S(\mathcal{X})=S(\operatorname{Ad}_{g}(\mathcal{X}))$ where
$$
S= \begin{cases}\operatorname{sgn}\left(\alpha_{2}\right) ; & \text { if } \alpha_{1}=0 \text{,} \\
0 & \text { otherwise. }\end{cases}
$$
\end{lem}
\begin{proof}
The proof is similar to the proof of the above Lemma.
    \end{proof}
\begin{lem}
$\mathcal{X}=\sum_{i=1}^{5} \alpha_{i} X_{i}$ with $\alpha_{1}=\alpha_{4}=\alpha_{3}=0$. Suppose $\operatorname{Ad}\left(\exp \left(\varepsilon X_{i}\right) \mathcal{X}\right)=\sum_{j=1}^{5} \alpha_{j}^{*} X_{j}$ also satisfies $\alpha_{1}^{*}=\alpha_{4}^{*}=\alpha_{3}^{*}=0$, for each $i=1,2,3,4,5$. Then $T^{*}=T$, i.e., $T(\mathcal{X})=T(\operatorname{Ad_{g}}(\mathcal{X}))$, where
$$
T= \begin{cases}\operatorname{sgn}\left(\alpha_{5}\right) & ;\text { if } \alpha_{1}=\alpha_{4}=\alpha_{3}=0 \text{,} \\ 0 &; \text { otherwise. }\end{cases}
$$
\end{lem}
\begin{proof}
     The proof is similar as discussed in Lemma \ref{lem_1}.
\end{proof}
\subsection{Construction of Optimal System}
\quad \quad To build an optimal system of Eq. \eqref{u} using the approach mentioned in \cite{hu2015direct}, we suppose 
$S=\sum_{i=1}^{5} \alpha_{i} X_{i} \hspace{0.2cm} \text{and} \hspace{0.2cm} M=\sum_{j=1}^{5} b_{i} X_{i}$ as two elements of Lie algebra $L^{5}$. The model's adjoint transformation equation is given by
\begin{equation}\label{urv}
\left(b_{1}, b_{2}, b_{3}, b_{4}, b_{5}\right)=\left(\alpha_{1}, \alpha_{2}, \alpha_{3}, \alpha_{4}, \alpha_{5}\right) \cdot \mathcal{A} \text{.}
\end{equation}
Further, $\mathcal{A}\left(\varepsilon_{1}, \varepsilon_{2},\varepsilon_{3} ,\varepsilon_{4}, \varepsilon_{5}\right)$ gives a transformed $\mathcal{X}$ as follows:
\begin{equation}\label{424}
\begin{split}
\mathcal{A}\left(\varepsilon_{1}, \varepsilon_{2}, \varepsilon_{3} ,\varepsilon_{4},\varepsilon_{5}\right) \cdot \mathcal{X}=& \alpha_{1} X_{1}+\left(-\alpha_{1} \varepsilon_{2} +\alpha_{2} \right) e^{\varepsilon_{1}} X_{2}+\alpha_{3} X_{3}+ \left(-\gamma \alpha_{1} \sigma_{1}+\alpha_{3} \sigma_{2}+\alpha_{4} \cos{\varepsilon_{3}}+\alpha_{5} \sin{\varepsilon_{3}}\right) e^{\gamma \varepsilon_{1}} X_{4}\\
& +\left(\gamma \alpha_{1} \sigma_{2}+\alpha_{3} \sigma_{1}-\alpha_{4} \sin{\varepsilon_{3}}+\alpha_{5}\cos{\varepsilon_{3}}\right) X_{5}\text{.} 
    \end{split}
    \end{equation}
By definition of $\mathcal{X}$,  $A\left(\varepsilon_{1}, \varepsilon_{2}, \varepsilon_{3}, \varepsilon_{4}, \varepsilon_{5}\right) \cdot \mathcal{X}$ produces one-dimensional Lie algebras for any $\varepsilon_{1}, \varepsilon_{2}, \varepsilon_{3}, \ldots \varepsilon_{5}$. This gives us the freedom to select different values of $\varepsilon_{i}$ to demonstrate the equivalence class of $\mathcal{X}$, which should be considerably simpler than $\mathcal{X}$. It is important to mention that the fundamental invariants that were obtained in this study are of degree one. Therefore, we evaluate the subsequent four cases in accordance with the invariants' degree:
$$
\begin{aligned}
\Lambda_{1}=\left\{\alpha_{1}=1,
\alpha_{3}=k\right\} \text{,} \quad \Lambda_{2}=\left\{\alpha_{1}=0,
\alpha_{3}=1\right\}\text{,} \quad \Lambda_{3}=\left\{\alpha_{1}=1, \alpha_{3}=0\right\} \text{,}\quad \Lambda_{4}=\left\{\alpha_{1}=0, \alpha_{3}=0\right\} \text {, }\hspace{0.1cm} \text{with} \hspace{0.1cm} k \neq 0\text{.}
\end{aligned}
$$
\subsection{Case 1}
$\alpha_{1}=1\text{,} \hspace{0.1cm} \alpha_{3}=k \neq 0\text{.}$ Choosing a representative element
$
\hat{\mathcal{X}}=X_{1}+\alpha_{2} X_{2}+k X_{3}+\alpha_{4} X_{4}+\alpha_{5} X_{5} \text {,} 
$
and putting $b_{i}=0\text{,} \hspace{0.2cm} \text{for} \hspace{0.2cm} i=2,4,5 \hspace{0.2cm} \text{and} \hspace{0.2cm} b_{1}=1\text{,} \hspace{0.2cm} b_{3}=k$ \hspace{0.1cm} in Eqs. \eqref{urv} and \eqref{424} we get the solution as
$$
\varepsilon_{2}=\alpha_{2} \hspace{0.2cm} \text{,}  \hspace{0.2cm} \varepsilon_{4}=\frac{\gamma \alpha_{4}-k \alpha_{5}}{k^{2}+\gamma ^{2}} \text{,}  \hspace{0.2cm} \varepsilon_{5}=\frac{\gamma \alpha_{5}+k \alpha_{4}}{k^{2}+\gamma^{2}}\text{.}
$$
Thus, the action of adjoint maps $\operatorname{Ad}\left(\exp \left(\varepsilon_{2} X_{2}\right)\right),\hspace{0.1cm} \operatorname{Ad}\left(\exp \left(\varepsilon_{4} X_{4}\right)\right)$ and $\operatorname{Ad}\left(\exp \left(\varepsilon_{5} X_{5}\right)\right)$ will eliminate the coefficients of $X_{2}, X_{4} \hspace{0.1cm} \text{and} \hspace{0.1cm} X_{5}$, respectively, from $\hat{\mathcal{X}}$ and 
we get $\hat{\mathcal{X}}=X_{1}+k X_{3}
$.
\subsection{Case 2} 
$\alpha_{1}=0\text{,} \hspace{0.1cm} \alpha_{3}=1\text{.}$ Choosing a representative element
$
\hat{\mathcal{X}}=\alpha_{2} X_{2}+X_{3}+\alpha_{4} X_{4}+\alpha_{5} X_{5}\text{,}
$
and putting $b_{i}=0\text{,} \hspace{0.2cm} \text{for} \hspace{0.2cm} i=1,2,4,5\hspace{0.2cm} \text{and} \hspace{0.2cm} b_{3}=1$ in Eqs. \eqref{urv} and \eqref{424}, we get the solution as 
$$
\varepsilon_{5}=\alpha_{4} \hspace{0.2cm} \text{,} \hspace{0.2cm} \varepsilon_{4}=-\alpha_{5}\text{.}
$$
Thus, the action of adjoint maps $\operatorname{Ad}\left(\exp \left(\varepsilon_{4} X_{4}\right)\right)\hspace{0.1cm}\text{and} \hspace{0.1cm}\operatorname{Ad}\left(\exp \left(\varepsilon_{5} X_{5}\right)\right)$ will eliminate the coefficients of $X_{4} \hspace{0.2cm} \text{and} \hspace{0.2cm} X_{5}$, respectively, from $\hat{\mathcal{X}}$. Thus, $\hat{\mathcal{X}}=b_{2} X_{2}+X_{3}$. Suppose $b_{2} \neq 0$, then the group generated by $X_{1}$ scales the coefficient of $X_{2}$, i.e., if $b_{2}>0$, then $\hat{\mathcal{X}}$ is equivalent to ${\hat{\mathcal{X}}_{1}}=X_{2}+X_{3}$ and if $b_{2}<0$, then $\hat{\mathcal{X}}$ is equivalent to ${\hat{\mathcal{X}}_{2}}=-X_{2}+X_{3}$. If $b_{2}=0$, then $\hat{\mathcal{X}}$ is equivalent to $\hat{\mathcal{X}_{3}}=X_{3}$.
\subsection{Case 3} 
$\alpha_{1}=1\text{,} \hspace{0.1cm} \alpha_{3}=0\text{.}$ Choosing a representative element
$\hat{\mathcal{X}}=X_{1}+\alpha_{2} X_{2}+\alpha_{4} X_{4}+\alpha_{5} X_{5}$, putting $b_{i}=0\text{,} \hspace{0.2cm} \text{for} \hspace{0.2cm} i=2,3,4,5 \hspace{0.2cm} \text{and} \hspace{0.2cm} b_{1}=1$ in Eqs. \eqref{urv} and \eqref{424}, we get
$$
\varepsilon_{2}=\alpha_{2}\hspace{0.2cm}\text{,} \hspace{0.2cm} \varepsilon_{4}=\frac{\alpha_{4}}{\gamma} \hspace{0.2cm} \text{,} \hspace{0.2cm} \varepsilon_{5}=\frac{\alpha_{5}}{\gamma}\text{.}
$$
Thus, the adjoint action of maps $\operatorname{Ad}\left(\exp \left(\varepsilon X_{2}\right)\right) \hspace{0.1cm}\text{,} \hspace{0.1cm} \operatorname{Ad} \left(\exp \left(\varepsilon X_{4}\right)\right)\hspace{0.1cm}\text{and} \hspace{0.1cm} A d\left(\exp \left(\varepsilon X_{5}\right)\right)$ will  eliminate the coefficients of $X_{2}, X_{4} \hspace{0.1cm} \text{and} \hspace{0.1cm} X_{5}$  from $\hat{\mathcal{X}}$.
Thus, $\hat{\mathcal{X}} = X_{1}$.
\subsection{Case 4}
$\alpha_{1}=0 \text{,}\hspace{0.2cm} \alpha_{3}=0\text{.}$ Choosing a representative element $\hat{\mathcal{X}}=\alpha_{2} X_{2}+\alpha_{4} X_{4}+\alpha_{5} X_{5}$ and putting $b_{i}=0\text{,} \hspace{0.2cm} \text{for} \hspace{0.2cm} 1\leq i \leq 5$ in Eqs. \eqref{urv} and \eqref{424}, we get
$$
\hat{\mathcal{X}}=\alpha_{2} X_{2}+\alpha_{4} X_{4}+\alpha_{5} X_{5}\text{.}
$$
The results acquired in the preceding sections can be expressed in the following theorem:
\begin{thm}
The vector fields generate the optimal one-dimensional subalgebra of the Lie algebra $\mathbb{L}$ of the model.
$$
\begin{aligned}
&\mathcal{X}_{1,k}=X_{1}+k X_{3}\text{,}\\
&\mathcal{X}_{2,\mu}=\mu_{1} X_{2}+X_{3}\text{,}\\
&\mathcal{X}_{3}=X_{1}\text{,}\\
&\mathcal{X}_{4, a_2, a_4, a_5}=a_{2} X_{2}+a_{4} X_{4}+a_{5} X_{5}\text{,}
\end{aligned}
$$
where $\mu_{1}$ is a parameter with $\mu_{1} \in \{-1,0,1\}$. Also, $k$ is to be determined according to earlier mentioned in case $1$, and $a_{j}^{'}s$ are arbitrary constants $\forall \hspace{0.1cm} j=2,4,5 \text{.}$
\end{thm}
\begin{proof}
It suffices to prove the theorem by demonstrating that the $\mathcal{X}_i^{'} s$ are distinct from one another. To achieve this, we will construct a table illustrating the invariants in Table \ref{t4}. Analysis of the contents of Table \ref{t4} reveals that the subalgebras are different.
\renewcommand{\arraystretch}{1.0}
\begin{table}[H]
    \centering
    \begin{tabular}{|c|cccccccc|}
    \hline
      & K & M & N & P & Q & R & S & T\\\hline
       $\mathcal{X}_{1,1}(X_1+X_3)$  & $2 \gamma^2-1$ & 1 & 1 & 1 & 1 & 0 & 0 & 0\\
        $\mathcal{X}_{1,-1}(X_1-X_3)$ & $2 \gamma^2-1$ & 1 & $-1$ & 1 & 1 & 0 & 0 & 0\\
       $\mathcal{X}_{2,1}(X_2+X_3)$  & $-2$ & 0 & 1 & 1 & 1 & 0 & 1 & 0\\
        $\mathcal{X}_{2,0}(X_3)$ & $-2$ & 0 & 1 & 1 & 1 & 0 & 0 & 0\\
        $\mathcal{X}_{2,-1}(-X_2+X_3)$ & $-2$ & 0 & 1 & 1 & 1 & 0 & -1 & 0\\
        $\mathcal{X}_{3}(X_1)$ & $2 \gamma^2 +1$  & 1 & 0 & 1 & 1 & 0 & 0 & 0\\
       $\mathcal{X}_{4, 0, a_4, a_5; a_4, a_5 \neq 0}(a_4 X_4+ a_5 X_5)$  & 0 & 0 & 0 & 0 & 1 & 0 & 0 & 0\\
       $\mathcal{X}_{4, a_2, 0, a_5; a_2, a_5 \neq 0}(a_2 X_2+ a_5 X_5)$& 0 & 0 & 0 & 1 & 1 & 0 & $a_2$ & $a_5$\\
        $\mathcal{X}_{4, a_2, a_4, 0; a_2, a_4 \neq 0}(a_2 X_2+ a_4 X_4)$ & 0 & 0 & 0 & 1 & 1 & $a_4$ & $a_2$ & $0$\\
        $\mathcal{X}_{4, 0, 0, a_5; a_5 \neq 0}(a_5 X_5)$ & 0 & 0 & 0 & 0 & 1 & 0 & 0 & $a_5$\\
        $\mathcal{X}_{4, a_2, 0, 0; a_2 \neq 0}(a_2 X_2)$ & 0 & 0 & 0 & 1 & 0 & 0 & $a_2$ & 0\\
        $\mathcal{X}_{4, 0, a_4, 0; a_4 \neq 0}(a_4 X_4)$ & 0 & 0 & 0  & 0 & 1 & $a_4$ & 0 & 0 \\ \hline
    \end{tabular}
    \caption{Value of Invariant Function}
    \label{t4}
\end{table}
\end{proof}
\section{Group Invariant Solutions}
\quad \quad Constructing a one-dimensional optimal system reduces the complexity of finding group-invariant solutions by identifying key symmetries in the model. This approach simplifies the classification of solutions, revealing deeper insights into the equation's structure. Invariant solutions often correspond to significant physical or geometric phenomena. 
\subsection{Invariant Solution Using Optimal System}
\subsubsection{\texorpdfstring{$\mathcal{X}=X_{1}+k X_{3} \hspace{0.1cm} \text{,} \hspace{0.1cm} k\neq 0$}{X = X1 + k X3 ; k != 0}}\label{urvj}
The similarity variables can be derived by solving the characteristic equations
\begin{equation}
\frac{d x}{\frac{(\theta-2) x}{2(\theta-1)}+k y}=\frac{d y}{\frac{(\theta-2) y}{2(\theta-1)}-k x}=\frac{d t}{t}=\frac{d \phi}{\frac{-\phi}{(\theta-1)}} \text{.}
\end{equation}
Then, we get
\begin{equation}
X= t^{\frac{-(\theta-2)}{2(\theta-1)}}[x \sin (k \ln (t))+y \cos (k \ln (t))]\text{,}\quad
Y=t^{\frac{-(\theta-2)}{2(\theta-1)}}[x \cos (k \ln (t))-y \sin (k \ln (t))]\text{,}\quad
\phi=t^{-\frac{1}{(\theta-1)}}\cdot F(X, Y)\text{.}
\end{equation}
Here, $X$ and $Y$ are two new independent variables and  $F$ is a new dependent variable. Therefore, the  Eq. \eqref{u} is converted into
\begin{equation}\label{r}
\begin{split}
&2 h(\theta-1) F^{\theta}-4(\theta-1)\left[F_{X}^{2}+F_{Y}^{2}\right]+2 k(\theta-1) \left[Y F_{X}-X F_{Y}\right]-\\
&(\theta-2){\left[X F_{X}+Y F_{Y}\right]-4\left[(\theta-1) \left(F_{X X}+F_{Y Y}\right)+\frac{1}{2}\right] F=0}\text{.} 
\end{split}
\end{equation}
Again, reducing Eq. \eqref{r} by point symmetries, the following vector fields are found to span the symmetry group of Eq. \eqref{r}
\begin{equation}
\xi_{X}=-d_{1} Y, \quad \xi_{Y}=d_{1} X, \quad \xi_{F}=0\text{,}
\end{equation}
where $d_{1}$ is an arbitrary constant.
It leads to the following characteristic equations:
\begin{equation}
\frac{d X}{-d_{1} Y}=\frac{d Y}{d_{1} X}=\frac{d F}{0}\text{,}
\end{equation}
which gives $F(X, Y)=G(\lambda)$, where $\lambda=X^{2}+Y^{2}$, and $G(\lambda)$ is a similarity function of $\lambda$. Thus, the second reduction of the model \eqref{u} by the similarity transform gives:
\begin{equation}\label{b}
-h(\theta-1) G^{\theta}+8 \lambda(\theta-1) G^{\prime 2}+8(\theta-1) G G^{\prime}+\lambda(\theta-2) G^{\prime}+8  \lambda(\theta-1) G G^{\prime \prime}
+G=0 \text{.}
\end{equation}
For $h=0$, we find the solution for the Eq. \eqref{b}.
\begin{equation}\label{z}
8 \lambda(\theta-1) G^{\prime 2}+8(\theta-1) G G^{\prime}+\lambda(\theta-2) G^{\prime}+ 8 \lambda(\theta-1) G G^{\prime \prime}
+G=0 \text{.}
\end{equation}
Particular solution for Eq. \eqref{z} is as follows:
\begin{equation}
    G(\lambda)=-\frac{\lambda}{16}\text{.}
\end{equation}
So, the invariant solution for the model \eqref{u} is as follows: 
\begin{equation}\label{533}
    \phi(x,y,t)=-\frac{x^{2}+y^{2}}{16 t} \text{.}
\end{equation}
\begin{figure}[H]
\begin{minipage}{.3\textwidth}
    \subfloat[3D profile at $t=10$]{\includegraphics[width=\textwidth]{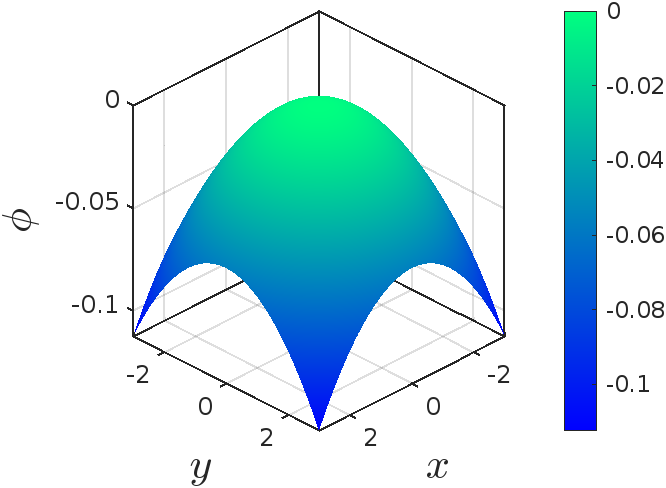}\label{fig:sub_a}}
\end{minipage}
\hfill    
\begin{minipage}{.3\textwidth}
    \subfloat[Contour sketch]{\includegraphics[width=\textwidth]{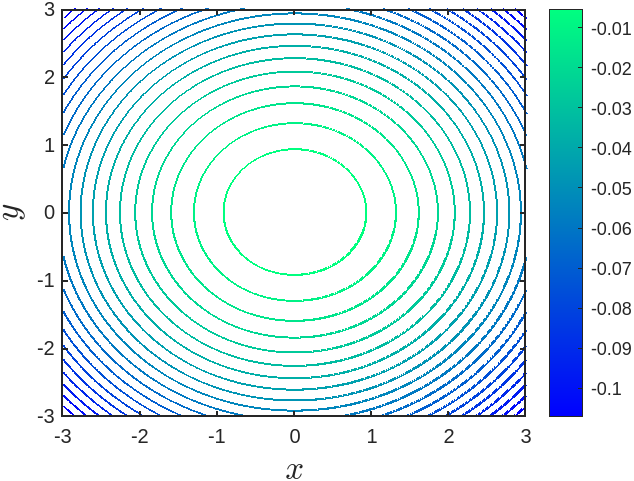}\label{fig:sub_b}}
\end{minipage}
\hfill    
\begin{minipage}{.3\textwidth}
    \subfloat[2D sketch for distinct $t$ at $x=2$ ]{\includegraphics[width=\textwidth]{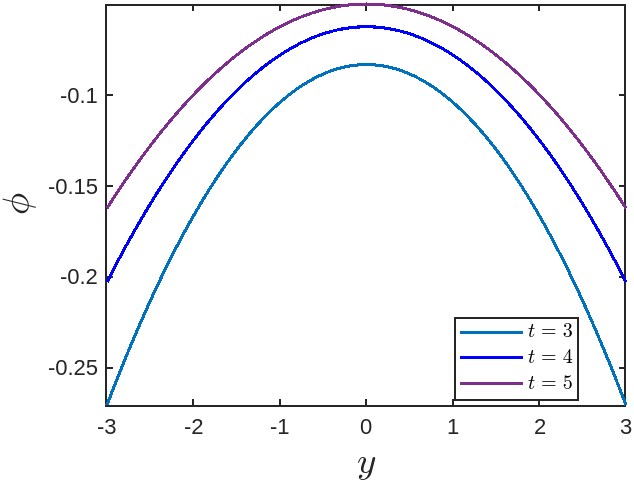}\label{fig:sub_c}}
\end{minipage}
    \caption{Solution profile of density $\phi$ to Eq. \ref{533}.}\label{fig:1}
\end{figure}

The solution of Eq. \eqref{533} is illustrated in Figure \ref{fig:1}. Figure \ref{fig:sub_a} illustrates that population density is concentrated near the origin and becomes less dense further away. In a physical sense, this is a scenario where the population is centered around a resource-rich area, with density diminishing as conditions become less favorable. Figure \ref{fig:sub_b} is the contour plot corresponding to the solution \eqref{533}. Figure \ref{fig:sub_c} reflects how a biological population might start with a more dispersed distribution and gradually concentrate in a central region over time.

In this instance, we set \(h=0 \), resulting in a negative population density. This scenario may occur when a contaminated lake exhibits zero fish reproduction (\( h = 0 \)), although environmental resistance (porous constant \(\theta\)) persists; negative population density could indicate severe environmental stress, leading the species toward extinction. This indicates unsustainable situations in which the habitat can no longer sustain the population.
\subsubsection{\texorpdfstring{$\mathcal{X}=b_{2} X_{2}+X_{3}$ \hspace{0.1cm} with \hspace{0.1cm} $b_{2} \in\{-1, 1\}$}{X = b2 X2 + X3 with b2 in {-1, 1}}}
The similarity variables can be determined by solving the characteristic equations.
\begin{equation}
\frac{d x}{y}=\frac{d y}{-x}=\frac{d t}{b_2}=\frac{d \phi}{0}\text{.}
\end{equation}
Then, we get
\begin{equation}
X=y \cos \left(\frac{t}{b_2}\right)+x \sin \left(\frac{t}{b_2}\right)\text{,} \quad
Y=x \cos \left(\frac{t}{b_2}\right)-y \sin \left(\frac{t}{b_2}\right)\text{,}\quad
\text{and}\hspace{0.2cm} \phi=F(X, Y)\text{.}
\end{equation}
Here, $X$ and $Y$ are the newly introduced independent variables, $b_{2} \neq 0$ and $F$ is a newly introduced dependent variable.
Therefore, the Eq. \eqref{u} is converted into
\begin{equation}\label{s}
2\left(F_{X}^{2}+F_{Y}^{2}\right)+2 F \left(F_{X X}+F_{Y Y}\right)-\frac{1}{b_2}\left(Y F_{X}-X F_{Y}\right)-h  F^{\theta}=0\text{.}
\end{equation}
Again, reducing Eq. \eqref{s} by point symmetries, the following vector fields are found to span the symmetry group of Eq. \eqref{s}
$$
\xi_{X}=d_{1} Y , \quad \xi_{Y}=-d_{1} X  , \quad \xi_{F}=0\text{,}
$$
where $d_{1}$ is an arbitrary constant. It leads to the following characteristic equations:
\begin{equation}
\frac{d X}{d_{1} Y}=\frac{d Y}{-d_{1} X}=\frac{d F}{0}\text{,}
\end{equation}
which gives $F(X, Y)=G(\lambda)$ with $\lambda=X^{2}+Y^{2}$, and $G(\lambda)$ is a similarity function of similarity variable $\lambda$.
Thus, the second reduction of the model by similarity transformation gives:
\begin{equation}
-h G^{\theta}+8\left[\lambda G^{\prime 2}+G\left(G^{\prime}+\lambda G^{\prime \prime}\right)\right]=0\text{.}
\end{equation}
The solution to the above ODE is given by
\begin{equation}
G(\lambda)  =\left[\frac{16}{h \lambda(\theta-2)^{2}}\right]^{\frac{1}{(\theta-2)}} \text{.}
\end{equation}
So, the invariant solution for the Model Eq. \eqref{u} is as follows: 
\begin{equation}\label{543}
\phi(x, y, t)  =\left[\frac{16}{h (\theta-2)^{2}\left(x^{2}+y^{2}\right)}\right]^{\frac{1}{\theta-2}}\text{.}
\end{equation}
\begin{figure}[H]
\centering
\begin{minipage}{.3\textwidth}
    \subfloat[3D graph for $\theta=0.055$ and $h=2$]{\includegraphics[width=\textwidth]{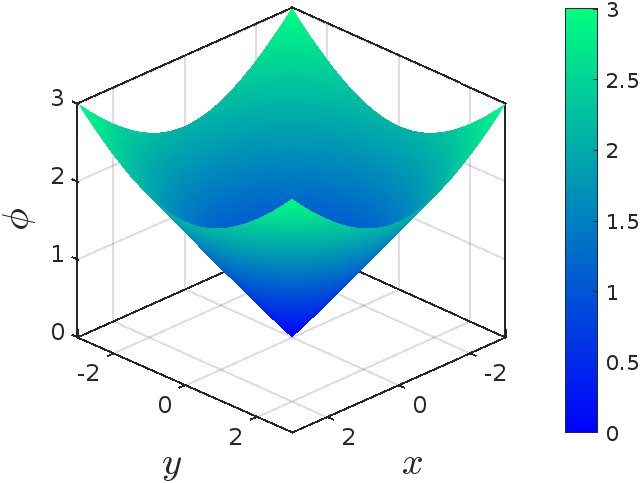}\label{fig:sub_2a}}
\end{minipage}
\hfill    
\begin{minipage}{.3\textwidth}
    \subfloat[3D graph for $\theta=0.555$ and $h=2$]{\includegraphics[width=\textwidth]{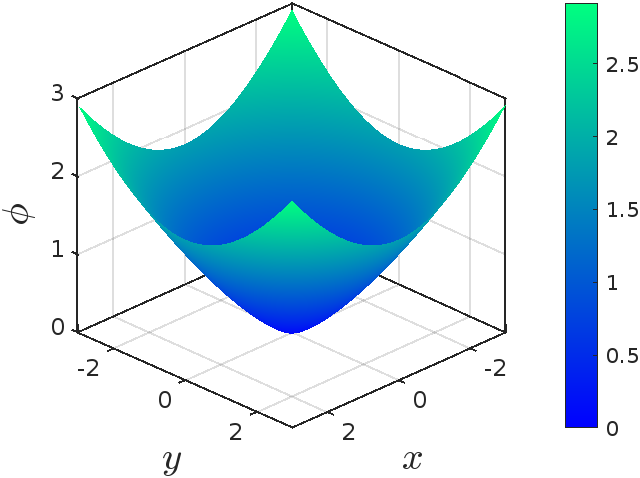}\label{fig:sub_2b}}
\end{minipage}
\hfill    
\begin{minipage}{.3\textwidth}
    \subfloat[3D graph for $\theta=0.955$ and $h=2$]{\includegraphics[width=\textwidth]{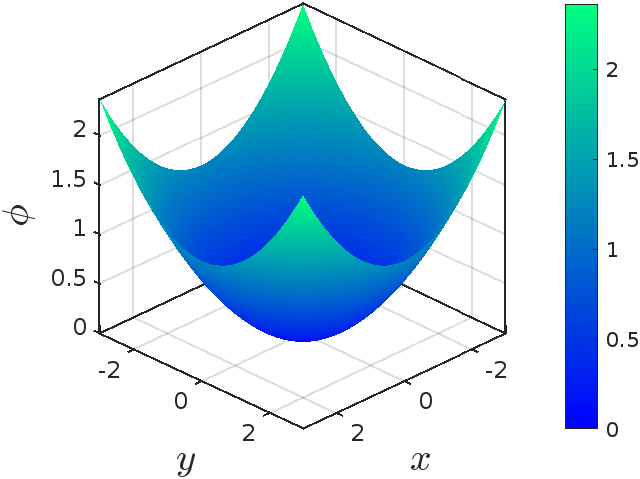}\label{fig:sub_2c}}
\end{minipage}
\quad
\begin{minipage}{.3\textwidth}
    \subfloat[Contour sketch for $\theta=0.055$ and $h=2$]{\includegraphics[width=\textwidth]{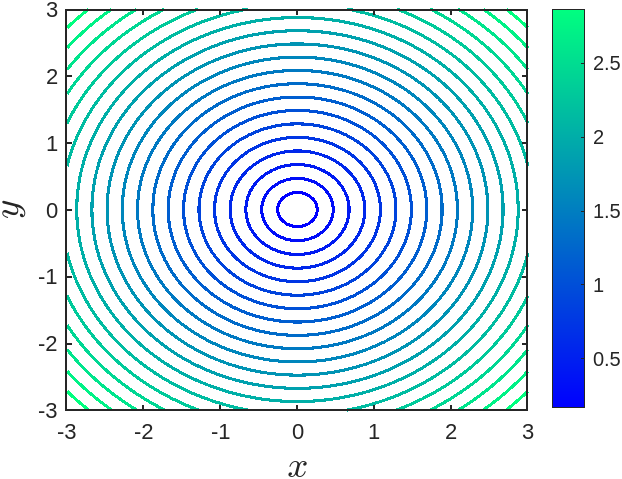}\label{fig:sub_2d}}
\end{minipage}
\hfill    
\begin{minipage}{.3\textwidth}
    \subfloat[Contour sketch for $\theta=0.555$ and $h=2$]{\includegraphics[width=\textwidth]{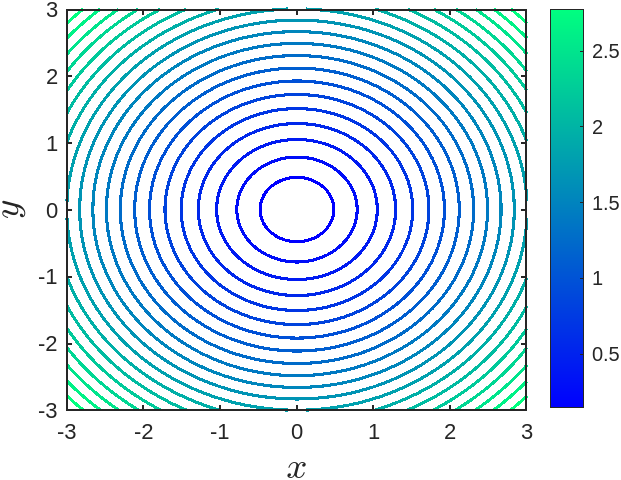}\label{fig:sub_2e}}
\end{minipage}
\hfill    
\begin{minipage}{.3\textwidth}
    \subfloat[Contour sketch for $\theta=0.955$ and $h=2$]{\includegraphics[width=\textwidth]{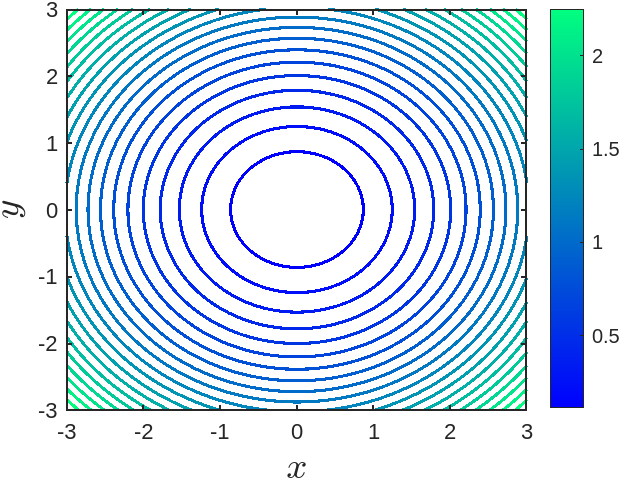}\label{fig:sub_2f}}
\end{minipage}
\quad
\vskip\floatsep
\begin{minipage}{.3\textwidth}
\centering
    \subfloat[2D sketch for distinct $\theta$ at $x=5$ and $h=2$]{\includegraphics[width=\textwidth]{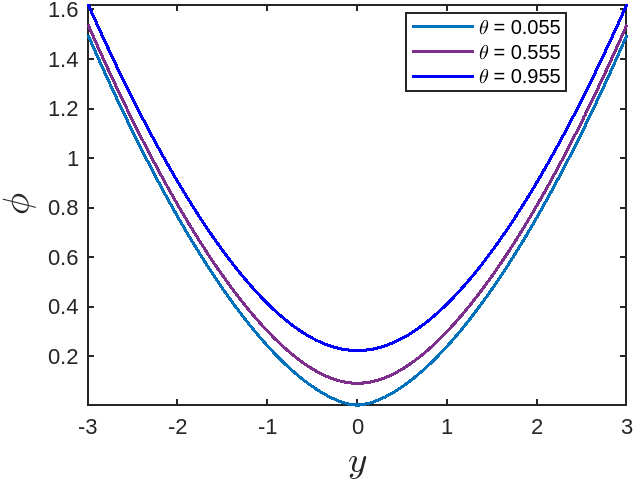}\label{fig:sub_2g}}
\end{minipage}
\quad
\begin{minipage}{.3\textwidth}
    \subfloat[2D sketch for distinct $h$ at $\theta=0.555$ and $x=5$]{\includegraphics[width=\textwidth]{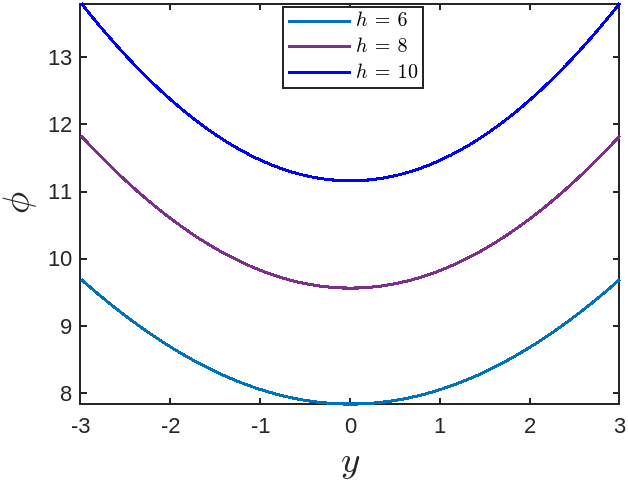}\label{fig:sub_2h}}
\end{minipage}
    \caption{Solution profile of density $\phi$ to Eq. \ref{543}.}\label{fig:2}
\end{figure}

For \( h=2 \), the solution of Eq. \eqref{543} is illustrated for various values of \( \theta \) in Figure \ref{fig:2}. Figure \ref{fig:sub_2a} shows that the small value of $\theta$ implies a very slow diffusion process (low porosity). The population is highly concentrated at the center, leading to a sharp peak. Figure \ref{fig:sub_2b} shows that the increase in $\theta$ means the medium is more porous, allowing the population to diffuse more easily. Figure \ref{fig:sub_2c} shows that with $\theta$ approaching 1, the medium becomes significantly more porous. Figures \ref{fig:sub_2d}, \ref{fig:sub_2e}, and \ref{fig:sub_2f} are the contour sketches. Figure \ref{fig:sub_2g} shows that the porous constant $\theta$ influences the spread or concentration of the population density. Smaller values of $\theta$ lead to a more concentrated population, while larger values result in a broader distribution. Figure \ref{fig:sub_2h} shows that as the population growth rate $h$ increases, the population density $\phi$ increases uniformly across all $y$ values. The symmetric shape around $y=0$ suggests an even distribution of density away from the center. 
\subsubsection{\texorpdfstring{$\mathcal{X}=X_{1}$}{X = X1}}
The similarity variables can be derived by solving the characteristic equations:
\begin{equation}
\frac{d x}{\frac{(\theta-2) x}{2(\theta-1)}}=\frac{d y}{\frac{(\theta-2) y}{2(\theta-1)}}=\frac{d t}{t}=\frac{d \phi}{\frac{-\phi}{(\theta-1)}}\text{.}
\end{equation}
Then, we get
\begin{equation}
\begin{split}
X=x \cdot t^{\frac{-(\theta-2)}{2(\theta-1)}} \text{,} \quad Y=y \cdot t^{-\frac{(\theta-2)}{2(\theta-1)}}\text{,}\quad \text{and} \hspace{0.2cm} \phi=t^{-\frac{1}{(\theta-1)}}\cdot F(X, Y)\text{.}
\end{split}
\end{equation}
Here, $X$ and $Y$ are newly introduced independent variables and  $F$ is a newly introduced dependent variable.
Therefore, the Eq. \eqref{u} is converted into
\begin{equation}\label{d}
-2 h(\theta-1) F^{\theta}+4(\theta-1)\left(F_{X X}+F_{Y Y}\right) F+2 F+4(\theta-1) \left(F_{X}^{2}+F_{Y}^{2}\right)+(\theta-2)\left(X F_{X}+Y F_{Y}\right)=0 \text{.}
\end{equation}
By applying point symmetries to Eq. \eqref{d} and simplifying it, we obtain the vector fields that form the symmetry group of Eq. \eqref{d}:
$$
\xi_{X}=-d_{1} Y, \quad \xi_{Y}=d_{1} X, \quad \xi_{F}=0\text{,}
$$
where $d_{1}$ is an arbitrary constant. It leads to the following characteristic equations
\begin{equation}
\frac{d X}{-d_{1} Y}=\frac{d Y}{d_{1} X}=\frac{d F}{0}\text{,}
\end{equation}
which gives $F(X, Y)=G(\lambda)$ with $\lambda=X^{2}+Y^{2}$, and $G(\lambda)$ is a similarity function of $\lambda$. Thus, the second reduction of the model by similarity transformation gives:
\begin{equation}\label{x}
-h(\theta-1) G^{\theta}+8 \lambda(\theta-1) G^{\prime 2}+(8(\theta-1) G+\lambda(\theta-2)) G^{\prime}+G\left(1+8 \lambda(\theta-1) G^{\prime \prime}\right)=0\text{.}
\end{equation}
The equation mentioned in Eq. \eqref{x} resembles the one derived in subsection \ref{urvj}. In this case, we obtain the identical outcome as previously addressed in section \ref{urvj}.  
\subsubsection{\texorpdfstring{$\mathcal{X}=\alpha_{2} X_{2}+\alpha_{4} X_{4}+\alpha_{5} X_{5}$}{X = α2 X2 + α4 X4 + α5 X5}}
The similarity variables can be derived by solving the characteristic equation.
\begin{equation}
\frac{d x}{\alpha_{4}}=\frac{d y}{\alpha_{5}}=\frac{d t}{\alpha_{2}}=\frac{d \phi}{0}\text{.}
\end{equation}
Then, we get
\begin{equation}
\begin{split}
X=x-\frac{\alpha_{4} t}{\alpha_{2}} \text{,}\quad Y=y-\frac{\alpha_{5} t}{\alpha_{2}} \hspace{0.1cm} \text{,} \hspace{0.1cm} \alpha_{2} \neq 0 \text{,} \hspace{0.1cm} \text{and} \hspace{0.2cm} \phi=F(X,Y)\text{.}
\end{split}
\end{equation}
Here, $X$ and $Y$ are newly introduced independent variables and $F$ is a newly introduced dependent variable.
Therefore, Eq. \eqref{u} is converted into
\begin{equation}\label{w}
2 \alpha_{2} F\left(F_{X X}+F_{Y Y}\right)+2 \alpha_{2}\left(F_{X}^{2}+F_{Y}^{2}-\frac{h F^{\theta}}{2}\right)+\alpha_{4}  F_{x}+\alpha_{5}  F_{y}=0\text{.}
\end{equation}
Again, reducing Eq. \eqref{w}  by point symmetries, the following vector fields are found to span the symmetry group of \eqref{w}
$$
\xi_{x}=d_{1}, \quad \xi_{y}=d_{2}, \quad  \xi_{F}=0\text{,}
$$
where $d_{1} \hspace{0.1cm} \text{and} \hspace{0.1cm} d_{2}$ are the arbitrary constants.
It leads to the following characteristic equations
\begin{equation}
\frac{d X}{d_{1}}=\frac{d Y}{d_{2}}=\frac{d F}{0}\text{,}
\end{equation}
which gives $F(X, Y)=G(\lambda)$,\hspace{0.1cm} where $\lambda=X-\frac{d_{1}}{d_{2}} Y$ with $d_{2} \neq 0$.
Thus, the second reduction of the model by similarity transformation gives:
\begin{equation}\label{557}
-\alpha_{2} h  G^{\theta}+2 \alpha_{2} \left(1+\left(\frac{d_{1}}{d_{2}}\right)^{2}\right) G^{\prime 2}+\left(\alpha_{4}-\frac{d_{1}}{d_{2}} \alpha_{5}\right) G^{\prime}+2 \alpha_{2} \left(1+\left(\frac{d_{1}}{d_{2}}\right)^{2}\right) G  G^{\prime \prime}=0 \text {. }
\end{equation}
So, the Eq. \ref{557} has solution as follows:
\begin{equation}
    G(\lambda)=\frac{\alpha_5 d_1-\alpha_4 d_2}{2 \alpha_2 (d_1^2+d_2^2)}d_2 \lambda \text{.}
\end{equation}
So, the invariant solution for the model Eq. \eqref{u} is as follows:
\begin{equation}\label{559}
    \phi(x,y,t)=\frac{\alpha_5 d_1-\alpha_4 d_2}{2 \alpha_2 (d_1^2+d_2^2)}\left( d_2 x-\frac{d_2 \alpha_4 t}{\alpha_2}-d_{1}y+\frac{d_{1} \alpha_5 t}{\alpha_2}\right)\text{.}
\end{equation}
\begin{figure}[H]
\centering
\begin{minipage}{.3\textwidth}
    \subfloat[3D profile at $t=1$]{\includegraphics[width=\textwidth]{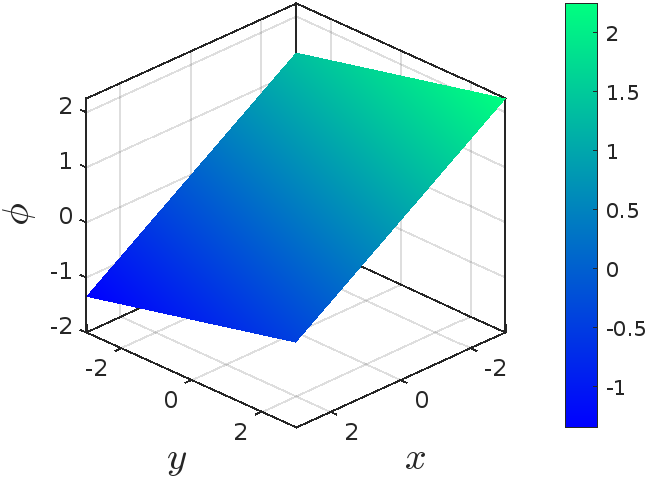}\label{fig:sub_3a}}
\end{minipage}
\quad 
\begin{minipage}{.3\textwidth}
    \subfloat[3D profile at $t=10$]{\includegraphics[width=\textwidth]{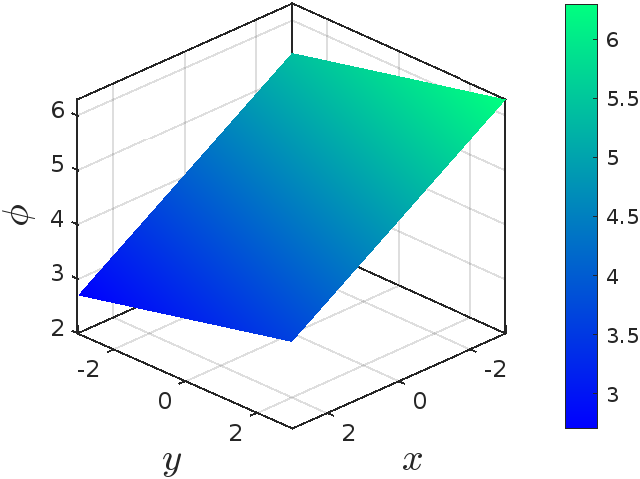}\label{fig:sub_3b}}
\end{minipage}
    \caption{Solution profile of density $\phi$ to Eq. \ref{559} for $\alpha_2 = 1, \alpha_4 = 1, \alpha_5 = 0, A_1 = 1,\hspace{0.1cm} \text{and}\hspace{0.1cm} A_2 = 3$.}\label{fig:3}
\end{figure}

We have pictured the solution of Eq. \ref{559} in Figures \ref{fig:3}. Figure \ref{fig:sub_4a} shows that at $t=1$, the population density is relatively low but varies across space. The graph shows a gradual change in density across the $x$ and $y$ axes, indicating that the population is spread out, with some areas slightly more populated than others.\\
Figure \ref{fig:sub_4b} shows that at $t=10$, the population density increases. This suggests that the population is starting to concentrate more in certain areas, leading to regions of higher density as seen in the steeper slopes of the plane.
\subsection{Invariant Solution Using Analysis of Parameters}
\begin{equation}
    \frac{dx}{\xi_1}= \frac{dy}{\xi_2}= \frac{dt}{\xi_3}= \frac{d\phi}{\xi_4}\text{.}
\end{equation}
\subsubsection{\texorpdfstring{$c_1=0$}{c1 = 0}}
The similarity variables can be derived by solving the characteristic equations
\begin{equation}
\frac{d x}{c_3 y +c_4}=\frac{d y}{-c_3 x+c_5}=\frac{d t}{c_2}=\frac{d \phi}{0}\text{.}
\end{equation}
Then, we get
\begin{equation}
X=\frac{(c_3 x-c_5) \sin (\frac{c_3 t}{c_2})+(c_3 y +c_4) \cos (\frac{c_3 t}{c_2})}{c_3}\text{,} \hspace{0.2cm} Y=\frac{(c_3 x-c_5) \cos (\frac{c_3 t}{c_2})-(c_3 y+c_4) \sin (\frac{c_3 t}{c_2})}{c_3}\text{,}\hspace{0.2cm}
\text{and} \hspace{0.2cm} \phi=F(X, Y)\text{.}
\end{equation}
Here, $X$ and $Y$ are two new independent variables and  $F$ is a new dependent variable. Therefore, Eq. \eqref{u} is converted into
\begin{equation}\label{rpa}
h F^{\theta}-2\left[F_{X}^{2}+F_{Y}^{2}\right]+\frac{c_3}{c_2}\left[Y F_{X}-X F_{Y}\right]-2\left[F_{X X}+F_{Y Y}\right] F=0\text{.} 
\end{equation}
Again, reducing Eq. \eqref{rpa} by point symmetries, the following vector fields are found to span the symmetry group of Eq. \eqref{rpa}:
\begin{equation}
\xi_{X}=d_{1} Y, \quad \xi_{Y}=-d_{1} X, \quad \xi_{F}=0\text{,}
\end{equation}
where $d_{1}$ is an arbitrary constant.
It leads to the following characteristic equations:
\begin{equation}
\frac{d X}{d_{1} Y}=\frac{d Y}{-d_{1} X}=\frac{d F}{0}\text{,}
\end{equation}
which gives $F(X, Y)=G(\lambda)$, where $\lambda=X^{2}+Y^{2}$, and $G(\lambda)$ is a similarity function of $\lambda$. Thus, the second reduction of the model \eqref{u} by the similarity transform gives:
\begin{equation}\label{1pa}
8 \lambda G^{\prime 2}+8 G G^{\prime}+ 8 \lambda G G^{\prime \prime}
-h G^\theta=0\text{.}
\end{equation}
Particular solution for Eq. \eqref{1pa} is as follows:
\begin{equation}
    G(\lambda)=\exp\left(\frac{\ln\left(\frac{16}{h (\theta - 2)^2\lambda}\right)}{\theta - 2}\right)\text{.}
\end{equation}
So, the invariant solution for the model \eqref{u} is as follows: 
\begin{equation}\label{561}
\phi(x,y,t)=\exp\left(\frac{\ln\left(\frac{16 c_3^2}{h (\theta - 2)^2 ((c_3 x-c_5)^2+(c_3 y+c_4)^2)}\right)}{\theta - 2}\right)\text{.}
\end{equation}
\begin{figure}[H]
\centering
\begin{minipage}{.3\textwidth}
    \subfloat[3D profile at $h=2$ and $\theta=0.555$]{\includegraphics[width=\textwidth]{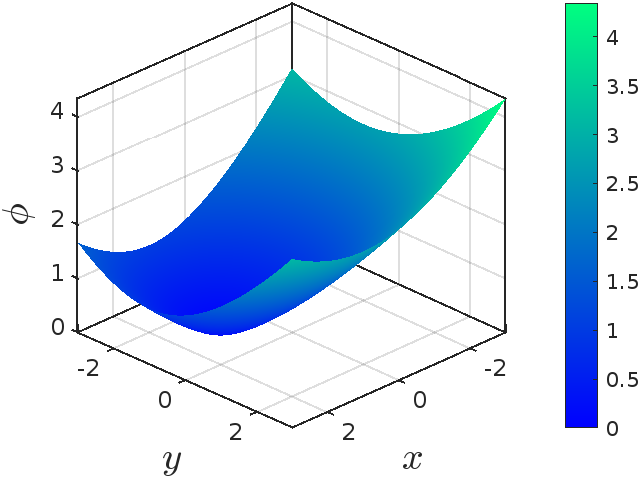}\label{fig:sub_4a}}
\end{minipage}
\quad 
\begin{minipage}{.3\textwidth}
    \subfloat[Contour sketch]{\includegraphics[width=\textwidth]{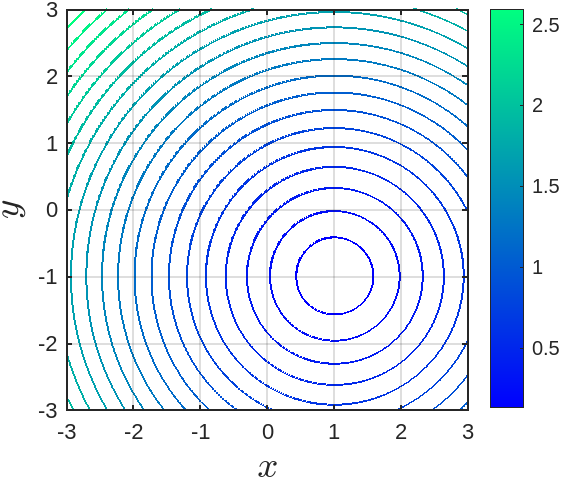}\label{fig:sub_4b}}
\end{minipage}
\quad \quad \quad \quad
\begin{minipage}{.3\textwidth}
    \subfloat[2D sketch for distinct $h$ at $\theta=0.555$ and $y=1$]{\includegraphics[width=\textwidth]{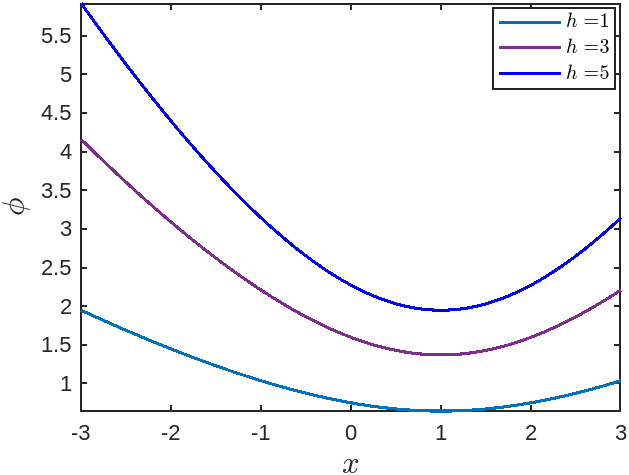}\label{fig:sub_4c}}
\end{minipage}
\quad 
\begin{minipage}{.3\textwidth}
    \subfloat[2D sketch for distinct $\theta$ at $h=2$ and $y=1$]{\includegraphics[width=\textwidth]{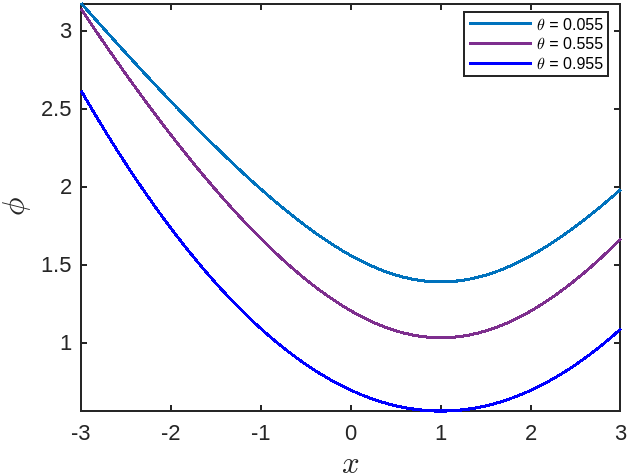}\label{fig:sub_4d}}
\end{minipage}
    \caption{Solution profile of density $\phi$ to Eq. \ref{561} for $c_3 = 1, c_4 = 1, c_5 = 1, h = 2,\hspace{0.1cm} \text{and}\hspace{0.1cm} \theta = 0.555$.}\label{fig:4}
\end{figure}

The solution of Eq. \eqref{561} is illustrated in Figure \ref{fig:4}. Figure \ref{fig:sub_4a} shows that the population density is highest in the central regions (depicted in red), gradually decreasing outward (shifting to blue). Figure \ref{fig:sub_4b} is the contour sketch for Eq. \eqref{561}. Figure \ref{fig:sub_4c} shows that a higher growth rate $h$ increases the overall population density, amplifying the peak. Figure \ref{fig:sub_4d} illustrates that a higher porous constant $\theta$ causes the population to spread out more, lowering the peak density. The porous constant $\theta$ influences how sharply the population spreads, while the growth rate $h$ controls the overall increase in population. This could represent natural phenomena like urban centers where resources or conditions are more favorable, leading to a dense population cluster. 
\subsubsection{\texorpdfstring{$c_2=0$}{c2 = 0}}
The similarity variables can be derived by solving the characteristic equations
\begin{equation}
\frac{d x}{c_1 \frac{(\theta-2) x}{2(\theta-1)}+c_3 y +c_4}=\frac{d y}{\frac{(\theta-2) y}{2(\theta-1)}-c_3 x+c_5}=\frac{d t}{c_1 t }=\frac{d \phi}{\frac{-c_1 \phi}{(\theta-1)}}\text{.}
\end{equation}
Then, we get
\begin{equation}
\begin{split}
& X=t^{\frac{(-\theta+2)}{2(\theta-1)}}\frac{[(u x+2 v (\theta-1))\cos (k \ln (t))-(u y + 2 w (\theta-1)) \sin (k \ln (t))]}{u},\\
& Y=t^{\frac{(-\theta+2)}{2(\theta-1)}}\frac{[(u y+2 w (\theta-1))\cos (k \ln (t))+(u x + 2 v (\theta-1)) \sin (k \ln (t))]}{u},\\
& \text{and} \hspace{0.1cm} \phi=t^{-\frac{1}{(\theta-1)}}\cdot F(X, Y)\text{,}
\end{split}
\end{equation}
where 
\begin{equation}
\begin{split}
&u=(\theta-2)^2 c_1^2+4 c_3^2 (\theta-1)^2\text{,}\quad v=c_4 (\theta-2) c_1-2 c_3 c_5 (\theta-1)\text{,}\quad w=c_5 (\theta-2) c_1+2 c_3 c_4 (\theta-1)\text{,} \quad\text{and} \hspace{0.2cm} k=\frac{c_3}{c_1}\text{.}
\end{split}
\end{equation}
Here, $X$ and $Y$ are two new independent variables and  $F$ is a new dependent variable. Therefore, the  Eq. \eqref{u} is converted into
\begin{equation}\label{rpa3}
\begin{split}
& h F^{\theta}-2\left[F_{X}^{2}+F_{Y}^{2}\right]+\frac{\left[X F_{X}+Y F_{Y}\right]}{2(\theta-1)}-2\left[F_{X X}+F_{Y Y}-\frac{1}{\theta-1}\right] F- (k Y+\frac{X}{2}) F_{X}+(k X-\frac{Y}{2})F_{y}=0\text{.} 
\end{split}
\end{equation}
Again, reducing Eq. \eqref{rpa3} by point symmetries, the following vector fields are found to span the symmetry group of Eq. \eqref{rpa3}
\begin{equation}
\xi_{X}=-d_{1} Y, \quad \xi_{Y}=d_{1} X, \quad \xi_{F}=0\text{,}
\end{equation}
where $d_{1}$ is an arbitrary constant. It leads to the following characteristic equations:
\begin{equation}
\frac{d X}{-d_{1} Y}=\frac{d Y}{d_{1} X}=\frac{d F}{0}\text{,}
\end{equation}
which gives $F(X, Y)=G(\lambda)$, where $\lambda=X^{2}+Y^{2}$, and $G(\lambda)$ is a similarity function of $\lambda$. Thus, the second reduction of the model \eqref{u} by the similarity transform gives:
\begin{equation}\label{bpa}
-h(\theta-1) G^{\theta}+8 \lambda(\theta-1) G^{\prime 2}+8(\theta-1) G G^{\prime}+\lambda(\theta-2) G^{\prime}+8  \lambda(\theta-1) G G^{\prime \prime}
+G=0
\end{equation}
For $h=0$, we find the solution for Eq. \eqref{bpa}.\\
\begin{equation}\label{zpa}
8 \lambda(\theta-1) G^{\prime 2}+8(\theta-1) G G^{\prime}+\lambda(\theta-2) G^{\prime}+ 8 \lambda(\theta-1) G G^{\prime \prime}
+G=0\text{.}
\end{equation}
Particular solution for Eq. \eqref{zpa} is as follows:
\begin{equation}
    G(\lambda)=-\frac{\lambda}{16}\text{.}
\end{equation}
So, the invariant solution for the model \eqref{u} is as follows: 
\begin{equation}\label{571}
    \phi(x,y,t)=-\frac{(u x+2 v (\theta-1))^{2}+(u y+2 w (\theta-1))^{2}}{16 u^2 t}\text{.}
\end{equation}
\begin{figure}[H]
\centering
\begin{minipage}{.3\textwidth}
    \subfloat[3D profile at $t=1$ and $\theta=0.5$]{\includegraphics[width=\textwidth]{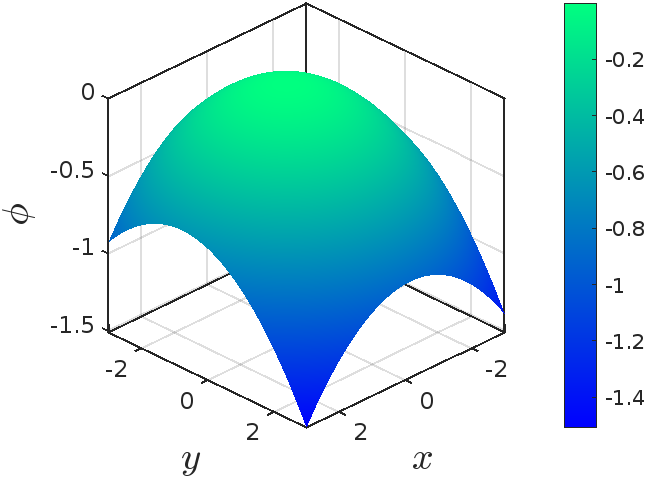}\label{fig:sub_5a}}
\end{minipage}
\quad
\begin{minipage}{.3\textwidth}
    \subfloat[Contour sketch]{\includegraphics[width=\textwidth]{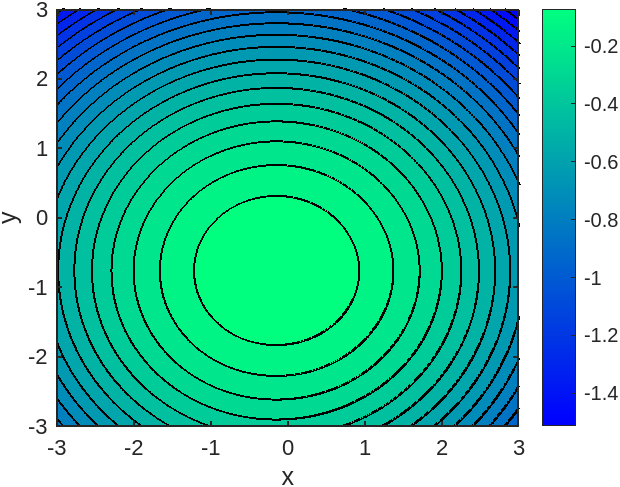}\label{fig:sub_5b}}
\end{minipage}
\quad \quad \quad \quad
\begin{minipage}{.3\textwidth}
    \subfloat[2D sketch for distinct $t$ at $\theta=0.5$ and $y=0$]{\includegraphics[width=\textwidth]{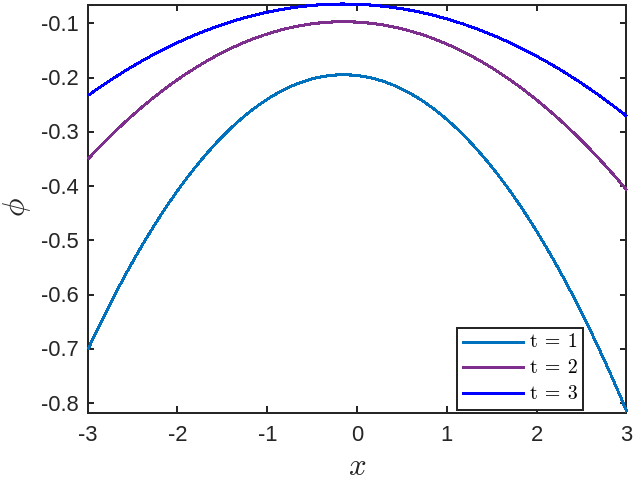}\label{fig:sub_5c}}
\end{minipage}
\quad
\begin{minipage}{.3\textwidth}
    \subfloat[2D sketch for distinct $\theta$ at $t=1$ and $y=0$]{\includegraphics[width=\textwidth]{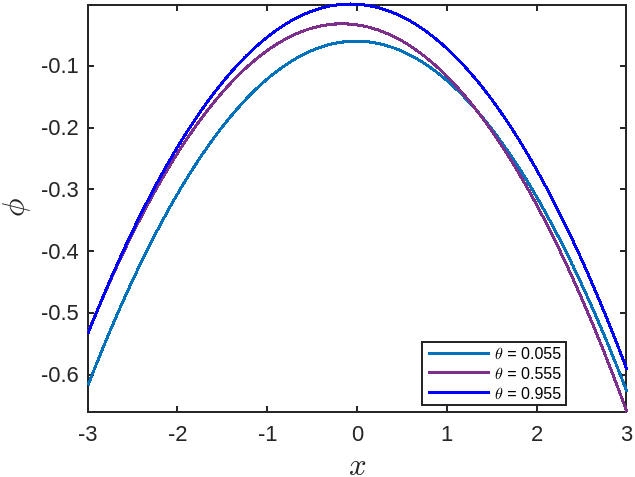}\label{fig:sub_5d}}
\end{minipage}
    \caption{Solution profile of density $\phi$ to Eq. \ref{571} for $c_1 = 1, c_3 = 1, c_4 = 1, c_5 = 1$.}\label{fig:5}
\end{figure}

The solution of Eq. \eqref{571} is illustrated in Figure \ref{fig:5}. Figure \ref{fig:sub_5a} shows how population density spreads through a porous medium over time. The result is a peak near the center with a gradual decline outward, indicative of typical diffusive behavior with anisotropic effects. Figure \ref{fig:sub_5b} is the contour sketch for Eq. \eqref{571}. Figure \ref{fig:sub_5c} shows the progression of the population density $\phi(x,t)$ along the $x$-axis as time progresses. The density starts highly concentrated at $t=1$ and gradually spreads out, becoming more uniform over time. Figure \ref{fig:sub_5d} illustrates the effect of the porous constant $\theta$ on the diffusion process. A low $\theta$ allows the substance (population, chemical, or heat) to diffuse more quickly and spread over a larger area, while a high $\theta$ restricts diffusion, keeping the substance more concentrated near its source. 
\subsubsection{\texorpdfstring{$c_3=0$}{c3 = 0}}
The similarity variables can be derived by solving the characteristic equations
\begin{equation}
\frac{d x}{\frac{c_1 (\theta-2) x}{2 (\theta-1)}+c_4}=\frac{d y}{\frac{c_1 (\theta-2) y}{2 (\theta-1)}+c_5}=\frac{d t}{c_1 t+c_2}=\frac{d \phi}{\frac{-c_{1} \phi}{(\theta-1)}}\text{.}
\end{equation}
Then, we get
\begin{equation}
\begin{split}
& X=\frac{\left( c_1 (\theta - 2) x + 2 c_4 (\theta - 1) \right) \left( c_1 t + c_2 \right)^{\frac{-\theta + 2}{2\theta - 2}}(\theta - 1)^{\frac{-\theta + 2}{2\theta - 2}}}{c_1 (\theta - 2)}\text{,}\\
& Y=\frac{\left( c_1 (\theta - 2) y + 2 c_5 (\theta - 1) \right)\left( c_1 t + c_2 \right)^{\frac{-\theta + 2}{2\theta - 2}}(\theta - 1)^{\frac{-\theta + 2}{2\theta - 2}}}{c_1 (\theta - 2)}
\text{,}\\
& \text{and} \hspace{0.1cm} \phi=(\theta-1)(c_1 t+c_2)^{\frac{-1}{\theta-1}}F(X, Y)\text{,}
\end{split}
\end{equation}
Here, $X$ and $Y$ are two new independent variables and  $F$ is a new dependent variable. Therefore, Eq. \eqref{u} is converted into
\begin{equation}\label{2pa}
h  F^{\theta}-2\left[F_{X}^{2}+F_{Y}^{2}\right]-\frac{c_1 (\theta-2)}{2}\left[X F_{X}+Y F_{Y}\right]-2\left[F_{X X}+F_{Y Y}\right] F -c_1 F=0\text{.} 
\end{equation}
Again, reducing Eq. \eqref{2pa} by point symmetries, the following vector fields are found to span the symmetry group of Eq. \eqref{2pa}:
\begin{equation}
\xi_{X}=d_{1} Y, \quad \xi_{Y}=-d_{1} X, \quad \xi_{F}=0\text{,}
\end{equation}
where $d_{1}$ is an arbitrary constant. It leads to the following characteristic equations:
\begin{equation}
\frac{d X}{d_{1} Y}=\frac{d Y}{-d_{1} X}=\frac{d F}{0}\text{,}
\end{equation}
which gives $F(X, Y)=G(\lambda)$, where $\lambda=X^{2}+Y^{2}$, and $G(\lambda)$ is a similarity function of $\lambda$. Thus, the second reduction of the model \eqref{u} by the similarity transform gives:
\begin{equation}\label{2pa2}
-h G^{\theta}+8 \lambda G^{\prime 2}+8 G G^{\prime}-c_1\lambda(\theta-2) G^{\prime}-c_1 G+8 \lambda G G^{\prime \prime}=0\text{.}
\end{equation}
Particular solution for Eq. \eqref{2pa} is as follows:
\begin{equation}
    G(\lambda)=\frac{c_1(\theta-1) \lambda}{16}\text{.}
\end{equation}
So, the invariant solution for the model \eqref{u} is as follows: 
\begin{equation}\label{579}
\phi(x,y,t)=\frac{\left[\left( c_1 (\theta - 2) x + 2 c_4 (\theta - 1) \right)^2+\left( c_1 (\theta - 2) y + 2 c_5 (\theta - 1) \right)^2\right]}{16 c_1 (c_1 t+c_2) (\theta-2)^2}(\theta-1)^{\frac{\theta}{\theta-1}}\text{.}
\end{equation}
\begin{figure}[H]
\centering
\begin{minipage}{.3\textwidth}
    \subfloat[3D profile at $t=1$ and $\theta=0.5$]{\includegraphics[width=\textwidth]{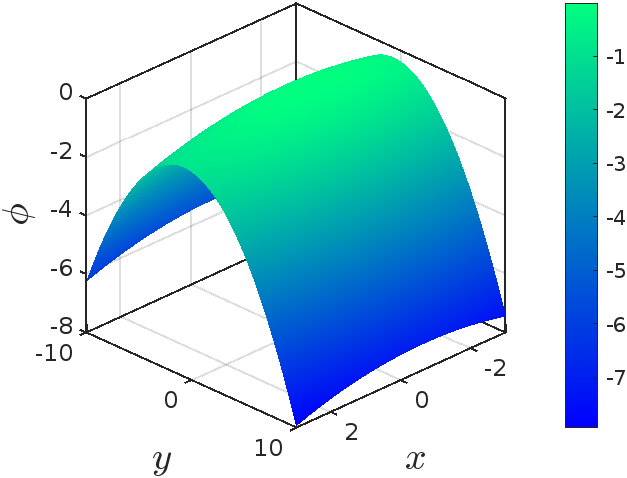}\label{fig:sub_6a}}
\end{minipage}
\hfill
\begin{minipage}{.3\textwidth}
    \subfloat[Contour sketch]{\includegraphics[width=\textwidth]{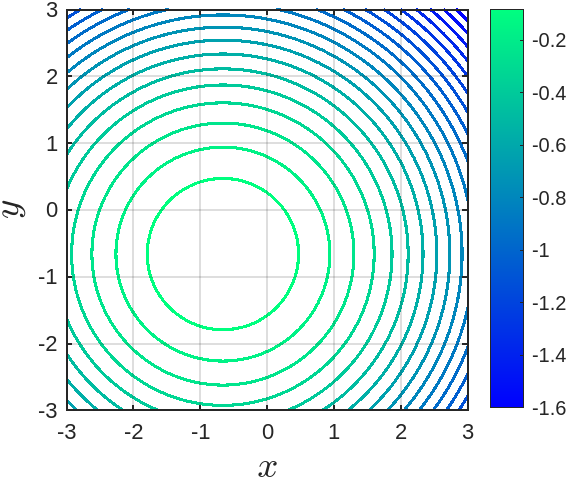}\label{fig:sub_6b}}
\end{minipage}
\hfill
\begin{minipage}{.3\textwidth}
    \subfloat[2D sketch for distinct $t$ at $\theta=0.9$ and $x=2$]{\includegraphics[width=\textwidth]{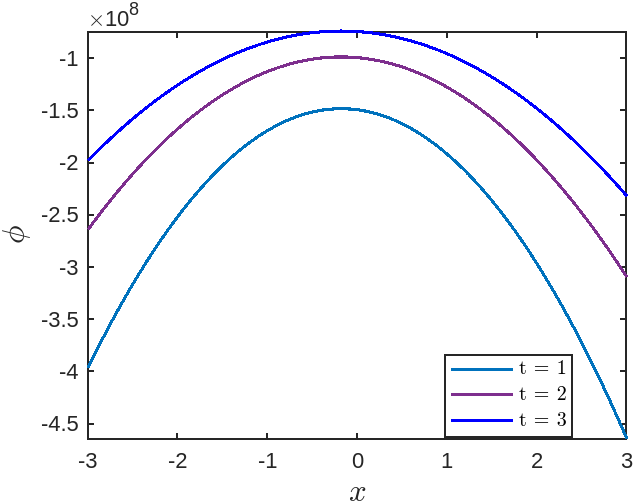}\label{fig:sub_6c}}
\end{minipage}
    \caption{Solution profile of density $\phi$ to Eq. \ref{579} for $c_1 = 1, c_2 = 1, c_4 = 1, c_5 = 1$.}\label{fig:6}
\end{figure}

The Solution of Eq. \eqref{579} is illustrated in Figure \ref{fig:6}. Figure \ref{fig:sub_6a} shows that the density distribution is non-uniform, with the central region having higher density, and this shape suggests population diffusion dynamics, where central populations disperse over time. Figure \ref{fig:sub_6b} is the contour sketch for Eq. \eqref{579}. Figure \ref{fig:sub_6c} reveals that as time progresses (from $t=1$ to $t=4$), the peak density increases and the distribution becomes more concentrated around $y=0$. This indicates a gathering or focusing of the population over time, with less dispersion occurring compared to earlier times.
\section{Nonlinear Self-Adjointness of Model}
\quad \quad Ibragimov \cite{ibragimov2011nonlinear} initially introduced the idea of nonlinear self-adjointness for a given system of partial differential equations (PDEs). This approach is highly efficient for systematically constructing conservation laws, regardless of whether the provided system of partial differential equations (PDEs) allows for a variational principle or not. In this section, we demonstrate that the given Eq. \eqref{u} is nonlinearly self-adjoint.
\begin{thm}
The given model Eq. \eqref{u} is nonlinearly self-adjoint.
\end{thm}
\begin{proof}
The adjoint equation corresponding to \eqref{u} can be written according to the procedure mentioned in \cite{ibragimov2011nonlinear}
\begin{equation}
S=\frac{\delta L}{\delta \phi}= h  \theta \phi^{\theta - 1} \Psi   - 2  \phi \frac{\partial^2 \Psi}{\partial x^2} - 2  \phi  \frac{\partial^2 \Psi}{\partial y^2} - \frac{\partial \Psi}{\partial t}\text{,}
\end{equation}
where $L$ is the formal Lagrangian defined as $L=\Psi(x,y,t) \Delta$, where\hspace{0.1cm} $\Delta=\phi_{t}-2 \phi_{x}^{2}-2 \phi  \phi_{x x}-2 \phi_{y}^{2}-2 \phi  \phi_{y y}+h \phi^{\theta}$ and $\Psi(x,y,t)$ is the dependent variable. The Euler-Lagrange operator $\frac{\delta }{\delta \phi}$ with respect to $\phi$ is as follows:
\begin{equation}
\frac{\delta}{\delta \phi} = \frac{\partial}{\partial \phi} - D_t \frac{\partial}{\partial \phi_t} - D_x \frac{\partial}{\partial \phi_x} - D_y \frac{\partial}{\partial \phi_y} + D_x^2 \frac{\partial}{\partial \phi_{xx}} + D_x D_y \frac{\partial}{\partial \phi_{xy}} + D_y^2 \frac{\partial}{\partial \phi_{yy}} + \cdots ,
\end{equation}
where $D_{x}$, $D_{y}$ and $D_{t}$ are total differential operators with respect to $x,y \hspace{0.1cm} \text{and} \hspace{0.1cm} t$ as given below:
\begin{equation}
   \begin{split}
      & D_x = \frac{\partial}{\partial x} + \phi_x \frac{\partial}{\partial \phi} + \phi_{xx} \frac{\partial}{\partial \phi_x} + \phi_{xy} \frac{\partial}{\partial \phi_y} + \phi_{xt} \frac{\partial}{\partial \phi_t} + \cdots ,\\
& D_y = \frac{\partial}{\partial y} + \phi_y \frac{\partial}{\partial \phi} + \phi_{yy} \frac{\partial}{\partial \phi_y} + \phi_{xy} \frac{\partial}{\partial \phi_x} + \phi_{yt} \frac{\partial}{\partial \phi_t} + \cdots ,\\
& D_t = \frac{\partial}{\partial t} + \phi_t \frac{\partial}{\partial \phi} + \phi_{tt} \frac{\partial}{\partial \phi_t} + \phi_{xt} \frac{\partial}{\partial \phi_x} + \phi_{yt} \frac{\partial}{\partial \phi_y} + \cdots .
   \end{split} 
\end{equation}
The given Eq. \eqref{u} is nonlinearly self-adjoint \cite{ibragimov2011nonlinear} if it satisfies the following condition: 
\begin{equation}\label{e61}
    S \mid_{(\Psi=G(x,y,t,\phi))} = \delta_{1} \Delta\text{,}
\end{equation}
where $G(x,y,t,\phi)$ is nonzero function, $\delta_{1}$ is unknown to be determined by equating coefficient of $\phi_{t}$ to zero and we obtain
\begin{equation}
    \delta_{1}=-\Psi_{\phi}\text{.}
\end{equation}
Substituting the value of $\delta_{1}$ into \eqref{e61}, one can obtain
\begin{equation}\label{e63}
\Psi=(c_{1} y+c_{2}) x+c_{3} y + c_{4}\text{,}
\end{equation}
where $c_{i}, i=1,2,3,4$ are the arbitrary constants. Hence, the given Eq. \eqref{u} is nonlinearly self-adjoint under \eqref{e63}.
\end{proof}
\section{Conservation Laws and Its Application}
\quad \quad Conservation laws are crucial in solving problems where certain physical properties remain unchanged over time. In mathematics, they are linked to the integrability of PDE systems and offer insights into the physical processes they model. These laws aid in analyzing the existence, uniqueness, stability, and development of numerical methods. Additionally, conservation laws can introduce nonlocal variables, leading to a nonlocally related PDE system for examining nonlocal symmetries \cite{sil2020nonlocal}.
\subsection{Costruction of Conservation Laws Using Nonlinear Self-Adjointness}
\quad \quad Ibragimov \cite{ibragimov2011nonlinear} provided an explicit formulation for the construction of conserved vectors for each associated symmetry. Ibragimov's method produces a specific subclass of conservation laws, whereas the multipliers method is the most general and produces all possible conservation laws. Given that the most general method by Anco appears to be unsuccessful for the governing model, we explicitly construct the conservation laws from each symmetry.\\
Let $\eta=(\eta^{x}, \eta^{y}, \eta^{t})$ be the conserved vector corresponding to $x,y$ and $t$, satisfying the conservation form
\begin{equation}
    D_{x} \eta^{x}+D_{y} \eta^{y}+D_{t} \eta^{t}=0\text{,}
    \end{equation}
and generated by the infinite symmetry 
\begin{equation}
    \mathbb{Z}=\xi_{1}\frac{\partial}{\partial x}+\xi_{2}\frac{\partial}{\partial y}+\xi_{3}\frac{\partial}{\partial t}+\xi_{4} \frac{\partial}{\partial \phi}\text{,}
\end{equation}
where the conserved vectors $\eta^{x},\eta^{y}$ and $\eta^{t}$ can be obtained from ideas developed in \cite{ibragimov2011nonlinear}
\begin{equation}\label{e766}
\begin{split}
& \eta^{x}=\xi_{1} L+\omega \frac{\partial L}{\partial \phi_{x}}-\omega D_{x} (\frac{\partial L}{\partial \phi_{xx}})+\frac{\partial L}{\partial \phi_{xx}}D_{x}(\omega)\text{,}\\
& \eta^{y}= \xi_{2} L+\omega \frac{\partial L}{\partial \phi_{y}}-\omega D_{y} (\frac{\partial L}{\partial \phi_{yy}})+\frac{\partial L}{\partial \phi_{yy}}D_{y}(\omega)\text{,}\\
& \eta^{t}=\xi_{3} L+\omega \frac{\partial L}{\partial \phi_{t}}\text{,}
\end{split}
\end{equation}
with $\omega=\xi_{4}-\xi_{1}\phi_{x}-\xi_{2}\phi_{y}-\xi_{3}\phi_{t}$. Furthermore, we evaluate the conserved vectors for the infinitesimal generators $X_{1}, X_{2}, X_{3}$ and $X_{4}$, as in the following cases:
\subsubsection{\texorpdfstring{$X_{1}= \gamma x \partial_{x}+ \gamma y \partial_{y}+t \partial_{t}- \tau \phi \partial_{\phi}$}{X1 = gamma x ∂x + gamma y ∂y + t ∂t - phi tau ∂phi}}
where $\frac{1}{\theta - 1}=\tau \hspace{0.2cm}\text{and} \hspace{0.2cm}\frac{(\theta - 2)}{2 (\theta - 1)}=\gamma$.
Using the vector $X_{1}$, we can obtain $\omega$ as follows:
\begin{equation}
\omega=-\tau \phi - \gamma x \phi_{x} - \gamma y \phi_{y} - t \phi_{t}\text{.}
\end{equation}
Using $\omega$ in Eq. \eqref{e766}, we can obtain $\eta^{x}, \eta^{y}$ and $\eta^{t}$ as follows:
\begin{equation}
\begin{split}            
\eta^{x}_{1}=& 2 \gamma x (c_1 x y + c_2 x + c_3 y + c_4) (\phi \phi_{xx}+ \phi_{x}^2 )+ 2(c_1 x y + c_2 x + c_3 y + c_4)(t \phi \phi_{xt} + \gamma y \phi \phi_{xy}\\
&+ t \phi_{t} \phi_{x} + \gamma y \phi_{y} \phi_{x} +2 \tau \phi \phi_{x})+ 2 \gamma (c_3 y+ c_4) \phi \phi_{x}-2 (c_1 y+ c_2)(2 t \phi \phi_{t} + \gamma y \phi \phi_{y}+2 \phi^2)\text{,}\\
\eta^{y}_{1}=&2 (c_1 x y + c_2 x + c_3 y + c_4)(\gamma (x \phi \phi_{xx}+ y \phi_{y}^2+ x \phi_{y} \phi_{x}+ y \phi \phi_{xy}+\phi \phi_{x})+t \phi_{t} \phi_{y}+ \tau \phi \phi_{y}\\
&+ t \phi \phi_{tx}+ \tau \phi \phi_{x})-2 \phi(c_1 x +c_3)(\gamma y \phi_{y}+t \phi_{t}+ \gamma x \phi_{x}+\tau \phi )\text{,}\\
\eta^{t}_{1}=&-(c_1 x y + c_2 x + c_3 y + c_4) \left( \tau \phi +\gamma x \phi_{x}+\gamma y \phi_{y}+t \phi_{t} \right)\text{.}
\end{split}
\end{equation}
\subsubsection{\texorpdfstring{$X_{2}=\partial_{t}$}{X2 = ∂t}}
Using the vector $X_{2}$, we can obtain $\omega$ as follows:
\begin{equation}
    \omega= - \phi_{t}\text{.}
\end{equation}
Using $\omega$ in Eq. \eqref{e766}, we can obtain $\eta^{x}, \eta^{y}$ and $\eta^{t}$ as follows:
\begin{equation}
\begin{split}
&\eta^{x}_{2}=(c_1 x y + c_2 x + c_3 y + c_4)(2 \phi_{x} \phi_{t} + 2  \phi \phi_{tx} ) -2(c_1 y + c_2) \phi\text{,}\\
&\eta^{y}_{2}=(c_1 x y + c_2 x + c_3 y + c_4)(2 \phi_{y} \phi_{t} + 2  \phi \phi_{ty} ) -2(c_1 x + c_3) \phi\text{,}\\
& \eta^{t}_{2}= -(c_1 x y + c_2 x + c_3 y + c_4) \phi_{t}\text{.}
    \end{split}
\end{equation}
\subsubsection{\texorpdfstring{$X_{3}=y \partial x - x \partial y$}{X3 = y ∂x - x ∂y}}
Using the vector $X_{3}$, we can obtain $\omega$ as follows:
\begin{equation}
    \omega=-y \phi_{x} +x \phi_{y}\text{.}
\end{equation}
Using $\omega$ in Eq. \eqref{e766}, we can obtain $\eta^{x}, \eta^{y}$ and $\eta^{t}$ as follows:
\begin{equation}
\begin{split}
& \eta^{x}_{3}=(c_1 x y + c_2 x + c_3 y + c_4) (-2 x \phi\phi_{xy} + 2 y \phi \phi_{xx}+2 y \phi_{x}^{2}-2 x \phi_{x} \phi_{y}) -2 \phi(c_3 y+c_4) \phi_{y}-2 y \phi (c_1 y + c_2) \phi_{x}\text{,}\\
& \eta^{y}_{3}=(c_1 x y + c_2 x + c_3 y + c_4) (-2 x \phi\phi_{yy} + 2 y \phi \phi_{xy}-2 x \phi_{y}^{2} +2 y \phi_{y} \phi_{x}) +2 \phi(c_2 x+c_4) \phi_{x}+2 x \phi (c_1 x + c_3) \phi_{y}\text{,}\\
& \eta^{t}_{3}=(c_1 x y + c_2 x + c_3 y + c_4)(-y \phi_{x}+x \phi_{y})\text{.}
    \end{split}
\end{equation}
\subsubsection{\texorpdfstring{$X_{4}=\partial x$}{X4 = ∂x}}
Using the vector $X_{4}$, we can obtain $\omega$ as follows:
\begin{equation}
    \omega=-\phi_{x}\text{.}
\end{equation}
Using $\omega$ in Eq. \eqref{e766}, we can obtain $\eta^{x}, \eta^{y}$ and $\eta^{t}$ as follows:
\begin{equation}
\begin{split}
&\eta^{x}_{4}=(c_1 x y + c_2 x + c_3 y + c_4)( 2 \phi_{x}^{2}+ 2 \phi \phi_{xx})-2 \phi \phi_{x} (c_1 y+ c_2)\text{,}\\
&\eta^{y}_{4}=2(c_1 x y + c_2 x + c_3 y + c_4)( \phi_{x} \phi_{y}+\phi \phi_{xy})-2\phi \phi_{x}(c_1 x+ c_3)\text{,}\\
&\eta^{t}_{4}=-(c_1 x y + c_2 x + c_3 y + c_4)\phi_{x}\text{.}
\end{split}
\end{equation}
\subsubsection{\texorpdfstring{$X_{5}=\partial y$}{X5 = ∂y}}
Using the vector $X_{5}$, we can obtain $\omega$ as follows:
\begin{equation}
    \omega=-\phi_{y}\text{.}
\end{equation}
Using $\omega$ in Eq. \eqref{e766}, we can obtain $\eta^{x}, \eta^{y}$ and $\eta^{t}$ as follows:
\begin{equation}
\begin{split}
&\eta^{x}_{5}=2(c_1 x y + c_2 x + c_3 y + c_4)( \phi_{x} \phi_{y}+\phi \phi_{xy})-2\phi \phi_{y}(c_1 y+ c_3)\text{,}\\
&\eta^{y}_{5}=(c_1 x y + c_2 x + c_3 y + c_4)(2 \phi_{y}^{2}+ 2 \phi \phi_{yy})-2 \phi \phi_{y} (c_1 x+ c_3)\text{,}\\
&\eta^{t}_{5}=-(c_1 x y + c_2 x + c_3 y + c_4)\phi_{y}\text{.}
    \end{split}
\end{equation}
\subsection{Applications}
\quad \quad From the nonlocally related partial differential equations (PDEs), one can derive nonlocal symmetries, exact solutions, and nonlocal conservation laws. However, generating these nonlocally related PDEs becomes significantly more complex for higher-dimensional cases (where \( n \geq 3 \)). Unlike the potential system for \( n = 2 \), the potential system for \( n \geq 3 \) is under-determined and therefore exhibits gauge freedom (for more details, refer to \cite{cheviakov2010multidimensional}). To remove this gauge freedom and ensure that the potential systems are completely defined, it is necessary to insert an extra equation that incorporates potential variables, which is referred to as a gauge constraint.
For instance, we will assume a conservation law that exhibits divergence in the $(x, y, z)$ space.

$$
\text{div}\hspace{0.1cm}\mathbb{T}=T^{1}_{x}+T^{2}_{y}+T^{3}_{z},
$$
where $\mathbb{T}(x,y,z)=(T^{1},T^{2},T^{3})$ denotes the flux vector. Consequently, a corresponding vector potential $J(x,y,z)=(J^{1},J^{2},J^{3})$ satisfies $\mathbb{T}=\text{curl}\hspace{0.1cm}J$. Therefore, the potential system can be formulated as follows:
\begin{equation}\label{787}
    \begin{split}
    &J^{3}_{y}-J^{2}_{z}=T^{1},\\
    &J^{1}_{z}-J^{3}_{x}=T^{2},\\
    &J^{2}_{x}-J^{1}_{y}=T^{3}\text{.}
    \end{split}
\end{equation}
To establish the desired potential systems and remove gauge freedom, it is crucial to include an additional equation involving potential variables, referred to as gauge constraints. The system \eqref{787} is inherently under-determined, requiring a gauge constraint to eliminate the gauge freedom and extract all possible solutions. Various gauge constraint options are available for this purpose, including the following:
\begin{enumerate}
    \item[1] Divergence (Coulomb) gauge: $\text{div} \hspace{0.1cm} J=J^{1}_{x}+J^{2}_{y}+J^{3}_{z}=0$,
    \item[2] Spatial gauge: $J^{i}=0,\hspace{0.2cm} i=1,2,3$,
    \item[3] Poincare gauge: $x J^{1}+y J^{2}+z J^{3}=0$.
\end{enumerate}
While dealing with PDE that involves the time coordinate t, our typical reliance lies in the (2+1) dimension as follows:
\begin{enumerate}
    \item[1] Lorentz gauge: $J^{1}_{t}-J^{2}_{x}-J^{3}_{y}=0$,
    \item[2] Cronstrom gauge:  $t J^{1}-x J^{2}-y J^{3}=0$.
\end{enumerate}
We list the possible potential systems for PDE \eqref{u} using conserved vectors from the above section.
\begin{equation}
\begin{split}
\mathcal{Y}_1:\hspace{0.1cm} J^{3}_{x}-J^{2}_{y}= &-(c_1 x y + c_2 x + c_3 y + c_4) \left( \tau \phi+\gamma x \phi_{x}+\gamma y \phi_{y}+t \phi_{t} \right) \text{,} \\
 J^{1}_{y}-J^{3}_{t}= &2 \gamma x (c_1 x y + c_2 x + c_3 y + c_4) (\phi \phi_{xx}+ \phi_{x}^2 )+ 2(c_1 x y + c_2 x + c_3 y + c_4)(t \phi \phi_{xt} + \gamma y \phi \phi_{xy}\\
&+ t \phi_{t} \phi_{x} + \gamma y \phi_{y} \phi_{x} +2 \tau \phi \phi_{x})+ 2 \gamma (c_3 y+ c_4) \phi \phi_{x}-2 (c_1 y+ c_2)(2 t \phi \phi_{t} + \gamma y \phi \phi_{y}+2 \phi^2)\text{,}\\
J^{2}_{t}-J^{1}_{x}=& 2 (c_1 x y + c_2 x + c_3 y + c_4)(\gamma (x \phi \phi_{xx}+ y \phi_{y}^2+ x \phi_{y} \phi_{x}+ y \phi \phi_{xy}+\phi \phi_{x})+t \phi_{t} \phi_{y}+ \tau \phi \phi_{y}\\
&+ t \phi \phi_{tx}+ \tau \phi \phi_{x})-2 \phi(c_1 x +c_3)(\gamma y \phi_{y}+t \phi_{t}+ \gamma x \phi_{x}+\tau \phi ) \text{.}
\end{split}
\end{equation}
\begin{equation}
    \begin{split}
\mathcal{Y}_2:\hspace{0.1cm}  J^{3}_{x}-J^{2}_{y}=&-(c_1 x y + c_2 x + c_3 y + c_4)\phi_{t},\\
J^{1}_{y}-J^{3}_{t}=&(c_1 x y + c_2 x + c_3 y + c_4)(2 \phi_{x} \phi_{t} + 2  \phi \phi_{tx} ) -2(c_1 y + c_2) \phi,\\
J^{2}_{t}-J^{1}_{x}=&(c_1 x y + c_2 x + c_3 y + c_4)(2 \phi_{y} \phi_{t} + 2  \phi \phi_{ty} ) -2(c_1 x + c_3) \phi\text{.}
\end{split}
\end{equation}
\begin{equation}
    \begin{split}
\mathcal{Y}_3:\hspace{0.1cm}  J^{3}_{x}-J^{2}_{y}=&(c_1 x y + c_2 x + c_3 y + c_4)(-y \phi_{x}+x \phi_{y}),\\
J^{1}_{y}-J^{3}_{t}=&(c_1 x y + c_2 x + c_3 y + c_4) (-2 x \phi\phi_{xy} + 2 y \phi \phi_{xx}+2 y \phi_{x}^{2}-2 x \phi_{x} \phi_{y}) -2 \phi(c_3 y+c_4) \phi_{y}-2 y \phi (c_1 y + c_2) \phi_{x},\\
J^{2}_{t}-J^{1}_{x}=&(c_1 x y + c_2 x + c_3 y + c_4) (-2 x \phi\phi_{yy} + 2 y \phi \phi_{xy}-2 x \phi_{y}^{2} +2 y \phi_{y} \phi_{x}) +2 \phi(c_2 x+c_4) \phi_{x}+2 x \phi (c_1 x + c_3) \phi_{y}\text{.}
\end{split}
\end{equation}
\begin{equation}
    \begin{split}
\mathcal{Y}_4:\hspace{0.1cm} J^{3}_{x}-J^{2}_{y}=&-(c_1 x y + c_2 x + c_3 y + c_4)\phi_{x},\\
J^{1}_{y}-J^{3}_{t}=&(c_1 x y + c_2 x + c_3 y + c_4)( 2 \phi_{x}^{2}+ 2 \phi \phi_{xx})-2 \phi \phi_{x} (c_1 y+ c_2),\\
J^{2}_{t}-J^{1}_{x}=&2(c_1 x y + c_2 x + c_3 y + c_4)( \phi_{x} \phi_{y}+\phi \phi_{xy})-2\phi \phi_{x}(c_1 x+ c_3)\text{.}
\end{split}
\end{equation}
\begin{equation}
    \begin{split}
\mathcal{Y}_5:\hspace{0.1cm} J^{3}_{x}-J^{2}_{y}=&-(c_1 x y + c_2 x + c_3 y + c_4)\phi_{y},\\
J^{1}_{y}-J^{3}_{t}=&2(c_1 x y + c_2 x + c_3 y + c_4)( \phi_{x} \phi_{y}+\phi \phi_{xy})-2\phi \phi_{y}(c_1 y+ c_3),\\
J^{2}_{t}-J^{1}_{x}=&(c_1 x y + c_2 x + c_3 y + c_4)(2 \phi_{y}^{2}+ 2 \phi \phi_{yy})-2 \phi \phi_{y} (c_1 x+ c_3)\text{.}
\end{split}
\end{equation}
In the future study, it will be quite interesting to investigate nonlocal symmetries \eqref{u} emerging from $\mathcal{Y}_1, \mathcal{Y}_2, \mathcal{Y}_3, \mathcal{Y}_4, \mathcal{Y}_5$ as well as inverse potential
systems and extract some new exact solutions.
\section{Conclusions}
\begin{sloppypar}
\quad \quad This paper presents a comprehensive study of the biological population model in porous media, employing mathematical techniques such as Lie symmetry analysis, optimal systems, and the construction of conservation laws. The analysis begins by identifying the Lie symmetries of the governing equations, enabling the simplification of the complex dynamics of population distribution in porous environments. These symmetries are used to derive invariant solutions, which provide a clearer understanding of how populations evolve under various environmental conditions, including regional heterogeneity, diffusion, and nonlinear interactions.

The research further develops an optimal system, categorizing the symmetries into equivalence classes to streamline the process of identifying group-invariant solutions. This step is crucial in understanding population behaviors over time, especially in scenarios where populations are influenced by factors such as resource availability, migration, and environmental constraints.

The conservation laws derived from the model offer additional insights into the underlying physical principles that govern population changes, such as mass and energy balance. The study also highlights how controlling key parameters, like the porous constant, can significantly influence population density and distribution, offering practical guidance for managing biological systems in real-world contexts.

Ultimately, this work provides a valuable contribution to both the theoretical and practical aspects of population dynamics in porous media. By bridging the gap between advanced mathematical methods and their biological applications, the paper equips researchers and practitioners with powerful tools to analyze, predict, and manage population growth and distribution in diverse ecological settings. This research can guide efforts in areas like habitat conservation, species management, and sustainability planning, helping to foster a balanced coexistence between biological populations and their environments.
\end{sloppypar}
\section{Future Recommendation}
\quad \quad The work offers valuable insights into the mechanisms and stabilizing of population dynamics, shedding light on the fundamental ecological processes involved in porous media. Subsequent research endeavors should encompass numerical simulations and visualization, empirical validation, extension to multivariate models, machine learning techniques for handling large datasets, generalization of conservation laws, and mathematical modeling to enhance the results of the paper and propel the field of mathematical biology and population dynamics forward.
\section{Acknowledgements}
\quad \quad The first author, Urvashi Joshi would like to express deepest gratitude to Aniruddha Kumar Sharma, the co-author, for his invaluable direction, support, and motivation throughout this research. She also acknowledges the financial support from the Ministry of Education (MoE). The second author, Aniruddha Kumar Sharma, acknowledges the financial support from the Council of Scientific and Industrial Research (CSIR) India, Grant/Award Number:09/0143(12918)/2021-EMR-I.
\section{Conflict of Interest}
According to the author's report, there is no conflict of interest associated with this article.
\bibliographystyle{elsarticle-num}
\bibliography{ref.bib}
\end{document}